\documentclass[10pt]{article}

\usepackage[T1]{fontenc}
\usepackage{lmodern}
\usepackage[utf8]{inputenc}
\usepackage{microtype}
\usepackage{framed}
\usepackage{listings}
\usepackage{vmargin}
\usepackage{setspace}
\usepackage{mathrsfs, mathenv}
\usepackage{amsmath, amsthm, amssymb, amsfonts, amscd}
\usepackage{graphicx}
\usepackage{epstopdf}
\usepackage[svgnames]{xcolor}
\usepackage{hyperref}
\hypersetup{citecolor=blue, colorlinks=true, linkcolor=black}
\usepackage{subcaption}
\usepackage{bbm}
\usepackage{cite}
\usepackage{verbatim}
\usepackage{pgfplots}
\usepackage{tikz}
\usetikzlibrary{arrows,decorations.pathmorphing,backgrounds,positioning,fit,matrix}
\usepackage{etoolbox}
\usepackage{color}
\usepackage{lipsum}
\usepackage{ifthen}
\usepackage[ruled, vlined]{algorithm2e}
\usepackage[title]{appendix}
\usepackage{accents}
\usepackage{xpatch}
\usepackage[]{algpseudocode}

\theoremstyle{plain}
\newtheorem{theorem}{Theorem}[section]
\newtheorem{corollary}[theorem]{Corollary}
\newtheorem{lemma}[theorem]{Lemma}
\newtheorem{proposition}[theorem]{Proposition}
\numberwithin{equation}{section}

\theoremstyle{definition}

\theoremstyle{remark}
\newtheorem{remark}[theorem]{Remark}
\newtheorem{assumption}[theorem]{Assumption}

\ifpdf
  \DeclareGraphicsExtensions{.eps,.pdf,.png,.jpg}
\else
  \DeclareGraphicsExtensions{.eps}
\fi

\usepackage{mathtools}
\mathtoolsset{showonlyrefs}

\usepackage{booktabs}

\usepackage{pgfplots}
\usepackage{tikz}
\usetikzlibrary{patterns,arrows,decorations.pathmorphing,backgrounds,positioning,fit,matrix}
\usepackage[labelfont=bf]{caption}
\usepackage{here}
\usepackage[font=normal]{subcaption}

\usepackage{enumitem}
\setlist[itemize]{leftmargin=.5in}
\setlist[enumerate]{leftmargin=.5in,topsep=3pt,itemsep=3pt,label=(\roman*)}

\newcommand*\samethanks[1][\value{footnote}]{\footnotemark[#1]}

\newcommand{\email}[1]{\href{#1}{#1}}
\newcommand{\TheTitle}{Eigenfunction martingale estimating functions and filtered data for drift estimation of discretely observed multiscale diffusions}
\newcommand{\TheAuthors}{A. Abdulle, G. A. Pavliotis, A. Zanoni}
\title{\TheTitle}
\author{Assyr Abdulle \thanks{ANMC, Institute of Mathematics, École Polytechnique Fédérale de Lausanne, 1015 Lausanne, Switzerland, \email{assyr.abdulle@epfl.ch}, \email{andrea.zanoni@epfl.ch}}
		\and Grigorios A. Pavliotis \thanks{Department of Mathematics, Imperial College London, London SW7 2AZ, UK, \email{g.pavliotis@imperial.ac.uk}}
		\and Andrea Zanoni \samethanks[1]
}
\date{}

\usepackage{amsopn}

\newcommand{\abs}[1]{\left\lvert#1\right\rvert}
\newcommand{\norm}[1]{\left\|#1\right\|}

\newcommand{\N}{\mathbb{N}}
\newcommand{\Z}{\mathbb{Z}}

\newcommand{\R}{\mathbb{R}}

\newcommand{\epl}{\varepsilon}
\newcommand{\diffL}{\mathcal{L}}
\newcommand{\defeq}{\coloneqq}
\newcommand{\eqdef}{\eqqcolon}

\newcommand{\E}{\operatorname{\mathbb{E}}}

\renewcommand{\d}{\mathrm{d}}
\newcommand{\dd}{\,\mathrm{d}}
\definecolor{shade}{RGB}{100, 100, 100}
\definecolor{bordeaux}{RGB}{128, 0, 50}

\renewcommand*{\dot}[1]{\accentset{\mbox{\large\bfseries .}}{#1}}

\usepackage[usestackEOL]{stackengine}

\definecolor{leg1}{RGB}{0,114,189}
\definecolor{leg2}{RGB}{217,83,25}
\definecolor{leg3}{RGB}{237,177,32}
\definecolor{leg4}{RGB}{126,47,142}
\definecolor{leg5}{RGB}{119,172,48}

\definecolor{leg21}{RGB}{62,38,169}
\definecolor{leg22}{RGB}{46,135,247}
\definecolor{leg23}{RGB}{55,200,151}
\definecolor{leg24}{RGB}{254,195,56}

\ifpdf
\hypersetup{
	pdftitle={\TheTitle},
	pdfauthor={\TheAuthors}
}
\fi

\begin{document}
	
\maketitle	

\vspace{-0.5cm}
\begin{center}
\it Dedicated to the memory of Assyr Abdulle
\end{center}

\begin{abstract} 
We propose a novel method for drift estimation of multiscale diffusion processes when a sequence of discrete observations is given. For the Langevin dynamics in a two-scale potential, our approach relies on the eigenvalues and the eigenfunctions of the homogenized dynamics. Our first estimator is derived from a martingale estimating function of the generator of the homogenized diffusion process. However, the unbiasedness of the estimator depends on the rate with which the observations are sampled. We therefore introduce a second estimator which relies also on filtering the data and we prove that it is asymptotically unbiased independently of the sampling rate. A series of numerical experiments illustrate the reliability and efficiency of our different estimators.
\end{abstract}

\textbf{AMS subject classifications.} 62F15, 65C30, 62M05, 74Q10, 35B27, 60J60, 76M50.

\textbf{Key words.} Langevin dynamics, diffusion process, homogenization, parameter estimation, discrete observations, eigenvalue problem, filtering, martingale estimators.

\section{Introduction}

Learning models from data is a problem of fundamental importance in modern applied mathematics. The abundance of data in many application areas such as molecular dynamics, atmosphere/ocean science make it possible to develop physics-informed data driven methodologies for deriving models from data \cite{RPK19, YMK21, ZLG19}. Naturally, most problems of interest are characterised by a very high dimensional state space and by the presence of many characteristic length and time scales. When it is possible to decompose the state space into the resolved and unresolved degrees of freedom, then one is usually interested in the derivation of a model for the resolved degrees of freedom, while treating the unresolved scales as noise. Clearly, these reduced models are stochastic, often described by stochastic differential equations (SDEs). The goal of this paper is to derive rigorous and systematic methodologies for learning coarse-grained models that accurately describe the dynamics at macroscopic length and time scales from noisy observations of the full, unresolved dynamics. We apply the proposed methodologies to simple models of fast/slow SDEs for which the theory of homogenization exists, that enables us to study the inference problem in a rigorous and systematic manner.

In many applications the available data are noisy, not equidistant and certainly not compatible with the coarse-grained model. The presence of observation noise and of the model-data mismatch renders the problem of learning macroscopic models from microscopic data highly ill-posed. Several examples from econometrics (market microstructure noise) \cite{AMZ06} and molecular dynamics show that standard algorithms, e.g., maximum likelihood or quadratic variation for the diffusion coefficient, are \emph{asymptotically biased} and they fail to estimate correctly the parameters in the coarse-grained model. In a series of earlier works this problem was studied using maximum likelihood techniques with subsampled data \cite{PaS07, PPS09}, methodologies based on the method of moments \cite{KPK13, KKP15, KPP15}, quadratic programming approaches \cite{CrV06} as well as Bayesian approaches \cite{AbD20, AGZ20}. We also mention the pioneering work on estimating the integrated stochastic volatility in the presence of market microstructure noise \cite{AMZ06, ZMP05}. In particular, in \cite{AiJ14} the authors analyse the correct interplay between the intensity of the microstructure noise and the optimal rates of convergence.

The main observation in \cite{PaS07, PPS09} is that when the maximum likelihood estimator (MLE) of the fast/slow system is evaluated at the full data, then the MLE becomes asymptotically biased; in fact, the original data are not compatible with the homogenized equation, and therefore data need to be preprocessed, for instance under the form of subsampling. On the other hand, when the MLE is evaluated at appropriately subsampled data, then it becomes asymptotically unbiased. Although this is an interesting theoretical observation (see also later developments in \cite{SpC13}), it does not lead to an efficient algorithm. The reason for this is that the performance of the estimator depends very sensitively on the choice of the sampling rate. In addition, the optimal sampling rate is not known and is strongly dependent on the problem under investigation. Furthermore, subsampling naturally leads to an increase in the variance, unless appropriate variance reduction methodologies are used.

In a recent work \cite{AGP20} we addressed the problem of lack of robustness of the MLE with subsampling algorithm by introducing an appropriate filtering methodology that leads to a stable and robust algorithm. In particular, rather than subsampling the original trajectory, we smoothed the data by applying an appropriate linear time-invariant filter from the exponential family and we modified the MLE by inserting the new filtered data. This new estimator was thus independent of the subsampling rate and also asymptotically unbiased and robust with respect to the parameters of the filter.

However, the assumption that the full path of the solution is observed is not realistic in most applications. In fact, in all real problems one can only obtain discrete measurements of the diffusion process. Hence, in this paper we focus on the problem of learning the coarse-grained homogenized model assuming that we are given discrete observations from the microscopic model. In this paper we use the martingale estimating functions that were introduced in \cite{BiS95}, where the authors study drift estimation for discrete observations of one-scale processes and show that estimators based on the discretized continuous-version likelihood function can be strongly biased. They therefore propose martingale estimating functions obtained by adjusting the discretized continuous-version score function by its compensator which leads to unbiased estimators. Moreover, in \cite{KeS99} a different type of martingale estimating function, which is dependent on the eigenvalues and eigenfunctions of the generator of the stochastic process, is introduced and asymptotic unbiasedness and normality are proved. Furthermore, another inference methodology that uses spectral information is proposed in \cite{CrV06a}. Their approach consists of inferring the drift and diffusion functions of a diffusion process by minimizing an objective function which measures how close the generator is to having a reference spectrum which is obtained from the time series through the construction of a discrete-time Markov chain. This idea has been further expanded in several directions in \cite{CrV11}.

In this paper, we propose a new estimator for learning homogenised SDEs from noisy discrete data that is based on the martingale estimators that were introduced in \cite{KeS99}. The main idea is to consider the eigenvalues and eigenfunctions of the generator of the homogenized process. This new estimator is asymptotically unbiased only if the distance between two consecutive observations is not too small compared with the multiscale parameter describing the fastest scale, i.e., if data are compatible with the homogenized model. Therefore, in order to obtain unbiased approximations independently of the sampling rate with which the observations are obtained, we propose a second estimator which, in addition to the original observations, relies also on filtered data obtained following the filtering methodology presented in \cite{AGP20}. We observe that smoothing the original data makes observations compatible with the homogenized process independently of the rate with which they are sampled and hence this second estimator gives a black-box tool for parameter estimation.

\subsection{Our main contributions}

The main goal of this paper is to propose new algorithms based on martingale estimating functions and filtered data for which we can prove rigorously that they are asymptotically unbiased and not sensitive with respect to e.g. the sampling rate and the observation error. In particular, we combine two main ideas:
\begin{itemize}
\item the use of martingale estimating functions for discretely observed diffusion processes based on the eigenvalues and the eigenfunctions of the generator of the homogenized process, which was originally presented for one-scale problems in \cite{KeS99};
\item the filtering methodology for smoothing the data in order to make them compatible with the homogenized model, which was introduced in \cite{AGP20}.
\end{itemize} 
We prove theoretically and observe numerically that the estimator without filtered data is asymptotically unbiased if:
\begin{itemize}
\item the observations are taken at the homogenized regime, i.e., the sampling rate is independent of the parameter measuring scale separation;
\item the observations are taken at the multiscale regime, i.e., the sampling rate is dependent on the fastest scale, and the sampling rate is bigger than the multiscale parameter.
\end{itemize}
Moreover, we show that the estimator with filtered data corrects the bias caused by a sampling rate smaller than the multiscale parameter and therefore it is asymptotically unbiased independently of the sampling rate.

\paragraph{Outline.} The rest of the paper is organized as follows. In Section 2 we present the Langevin dynamics and its corresponding homogenized equation and we introduce the two proposed estimators based on eigenvalues and eigenfunctions of the generator with and without filtered data. In Section 3 we present the main results of this work, i.e., the asymptotic unbiasedness of the two estimators, and in Section 4 we perform numerical experiments which validate the efficacy of our methods. Section 5 is devoted to the proof of the main results which are presented in Section 3. Finally, in the Appendix we show some technical results which are employed in the analysis and we explain some details about the implementation of the proposed methodology.

\section{Problem setting}

In this work we study the following class of multiscale diffusion processes. Consider the following two-scale SDE, observed over the time interval $[0,T]$
\begin{equation} \label{eq:SDE_MS}
\d X_t^\epl = -\alpha \cdot V'(X_t^\epl) \dd t - \frac1\epl p'\left(\frac{X_t^\epl}\epl\right) \dd t + \sqrt{2\sigma} \dd W_t,
\end{equation}
where $\epl>0$ describes the fast scale, $\alpha\in\R^M$ and $\sigma>0$ are respectively the drift and diffusion coefficients and $W_t$ is a standard one-dimensional Brownian motion. The functions $V\colon\R\to\R^M$ and $p\colon\R\to\R$ are the slow-scale and fast-scale parts of the potential and they are assumed to be smooth. Moreover, we also assume $p$ to be periodic with period $L>0$. We remark that our setting can be considered as a semi-parametric framework similar to the one of \cite{KPK13}. The components of the potential function $V$, in fact, can be viewed as basis functions for a truncated expansion (e.g., Taylor series or Fourier expansion) of the unknown slow-scale potential $V(\cdot;\alpha) \colon \R \to \R$, where the components of the unknown drift term $\alpha$ contain the generalized Fourier coefficients, i.e., 
\begin{equation}
V(x;\alpha) = \sum_{m=1}^M \alpha_m V_m(x).
\end{equation}
We also mention that assuming a parametric form for the potential $V$ is a technique usually employed in the statistics literature in order to regularize the likelihood function and obtain a parametric approximation of the actual MLE of $V$, which does not exist in general \cite{PSV09}.
\begin{remark}
For clarity of the presentation, we focus our analysis on scalar multiscale diffusions with a finite number of parameters in the drift that have to be learned from data. Nevertheless, we remark that all the following theory can be generalized to the case of multidimensional diffusion processes in $\R^d$, for which we provide further details in Appendix \ref{app:multidimensional} and an example in Section \ref{sec:IP}. However, the problem becomes more complex and computationally expensive from a numerical viewpoint and it can be prohibitive if the dimension $d$ is too large, since the methodology proposed in this paper requires the solution of the eigenvalue problem for the generator of a $d$-dimensional diffusion process.
\end{remark}
The theory of homogenization (see e.g. \cite[Chapter 3]{BLP78} or \cite[Chapter 18]{PaS08}) guarantees the existence of the following homogenized SDE whose solution $X_t^0$ is the limit in law of the solutions $X_t^\epl$ of \eqref{eq:SDE_MS} as random variables in $\mathcal C^0([0,T]; \R)$
\begin{equation} \label{eq:SDE_H}
\d X_t^0 = -A \cdot V'(X_t^0) \dd t + \sqrt{2\Sigma} \dd W_t,
\end{equation} 
where $A=K\alpha$, $\Sigma=K\sigma$. The coefficient $0<K<1$ has the explicit formula
\begin{equation} \label{eq:K_H}
K = \int_0^L (1 + \Phi'(y))^2 \, \mu(\d y) = \int_0^L (1 + \Phi'(y)) \, \mu(\d y),
\end{equation}
with 
\begin{equation} \label{eq:def_mu}
\mu(\d y) = \frac{1}{C_\sigma} e^{-p(y)/\sigma} \dd y, \quad\text{where} \quad C_\sigma = \int_0^L e^{-p(y)/\sigma} \dd y,
\end{equation}
and where the function $\Phi$ is the unique solution with zero-mean with respect to the measure $\mu$ of the differential equation
\begin{equation} \label{eq:cell_problem}
-p'(y)\Phi'(y) + \sigma \Phi''(y) = p'(y), \quad 0 \leq y \leq L,
\end{equation}
endowed with periodic boundary conditions. In particular, for one-dimensional diffusion processes we have
\begin{equation}
\Phi'(y) = \frac{L}{\widehat C_\sigma} e^{p(y)/\sigma} - 1, \quad \text{where} \quad \widehat C_\sigma = \int_0^L e^{p(y)/\sigma} \dd y,
\end{equation}
which implies 
\begin{equation}
K = \frac{L^2}{C_\sigma \widehat C_\sigma}.
\end{equation}
Our goal is to derive estimators for the homogenized drift coefficient $A$ based on multiscale data originating from \eqref{eq:SDE_MS}. In this work we consider the same setting as \cite{AGP20}, which is summarized by the following assumption.
\begin{assumption} \label{ass:dissipative_setting}
The potentials $p$ and $V$ satisfy
\begin{enumerate}
\item $p \in \mathcal C^\infty(\R) \cap L^\infty(\R)$ and is $L$-periodic for some $L > 0$,
\item\label{ass:regularity_diss} $V \in \mathcal C^\infty(\R;\R^M)$ and each component is polynomially bounded from above and bounded from below, and there exist $b_1,b_2 > 0$ such that
\begin{equation}
- b_1 + b_2 x^2 \le \alpha \cdot V'(x) x,
\end{equation}
\item $V'$ is Lipschitz continuous, i.e. there exists a constant $C > 0$ such that
\begin{equation}
\norm{V'(x) - V'(y)} \leq C\abs{x - y}.
\end{equation} 
\end{enumerate}
\end{assumption}
Let us remark that, under Assumption \ref{ass:dissipative_setting}, it has been proved in \cite{PaS07} that both processes \eqref{eq:SDE_MS} and \eqref{eq:SDE_H} are geometrically ergodic and their invariant measure has a density with respect to the Lebesgue measure. In particular, let us denote by $\varphi^\epl$ and $\varphi^0$ the densities of the invariant measures of $X_t^\epl$ and $X_t^0$, respectively defined by
\begin{equation} \label{eq:invariant_measure_Xe}
\varphi^\epl(x) = \frac{1}{C_{\varphi^\epl}} \exp \left( -\frac1\sigma \alpha \cdot V(x) - \frac1\sigma p\left( \frac x\epl \right) \right), \quad\text{where} \quad C_{\varphi^\epl} = \int_\R \exp \left( -\frac1\sigma \alpha \cdot V(x) - \frac1\sigma p\left( \frac x\epl \right) \right) \dd x,
\end{equation}
and
\begin{equation} \label{eq:invariant_measure_X0}
\varphi^0(x) = \frac{1}{C_{\varphi^0}} \exp \left( -\frac1\Sigma A \cdot V(x) \right), \quad\text{where} \quad C_{\varphi^0} = \int_\R \exp \left( -\frac1\Sigma A \cdot V(x) \right) \dd x.
\end{equation}
\begin{remark} \label{rem:initial_condition}
The value of the initial condition $X_0^\epl$ in the SDE \eqref{eq:SDE_MS} is important neither for the numerical experiments nor for the following analysis and can be chosen arbitrarily. In fact, the process $X_t^\epl$ is geometrically ergodic and therefore it converges to its invariant distribution with density $\varphi^\epl$ exponentially fast for any initial condition.
\end{remark}
\paragraph{Drift estimation problem.} Consider $N+1$ uniformly distributed observation times $0 = t_0 < t_1 < t_2 < \dots, < t_N = T$, set $\Delta = t_n - t_{n-1}$ and let $(X_t^\epl)_{t\in[0,T]}$ be a realization of the solution of \eqref{eq:SDE_MS}. We then assume to know a sample $\{ \widetilde X_n^\epl\}_{n=0}^N$ of the realization where $\widetilde X_n^\epl = X_{t_n}^\epl$ and we aim to estimate the drift coefficient $A$ of the homogenized equation \eqref{eq:SDE_H}. First, since we deal with discrete observations of stochastic processes, we employ martingale estimating functions based on eigenfunctions, which have already been studied for problems without a martingale structure in \cite{KeS99}. Second, by observing that if the time-step $\Delta$ is too small with respect to the multiscale parameter $\epl$, then the data could be compatible with the full dynamics rather than with the coarse-grained model, we also adopt the filtering methodology presented in \cite{AGP20}, which has been proved to be beneficial for correcting the behavior of the maximum likelihood estimator (MLE) in the setting of continuous observations.

\subsection{Martingale estimating functions based on eigenfunctions}

We first remark that a general theory for martingale estimating functions exists and is thoroughly outlined in \cite{BiS95}. They appear to be appropriate for multiscale problems due to their robustness properties. In this paper we develop martingale estimating functions based on the eigenfunctions of the generator of the process, since the theory of the eigenvalue problem for elliptic differential operators and the multiscale analysis of this eigenvalue problem are well developed. Let $\mathcal A \subset \R^M$ be the set of admissible drift coefficients for which Assumption \ref{ass:dissipative_setting}\ref{ass:regularity_diss} is satisfied. To describe our methodology we consider the solution $X_t(a)$ of the homogenized process \eqref{eq:SDE_H} with a generic parameter $a\in\mathcal A$ instead of the exact drift coefficient $A$:
\begin{equation} \label{eq:SDE_H_a}
\d X_t(a) = -a \cdot V'(X_t(a)) \dd t + \sqrt{2\Sigma} \dd W_t,
\end{equation} 
which, according to \eqref{eq:invariant_measure_X0}, has invariant measure 
\begin{equation} \label{eq:phia_def}
\varphi_a(x) = \frac{1}{C_{\varphi_a}} \exp \left( -\frac1\Sigma a \cdot V(x) \right), \quad\text{where} \quad C_{\varphi_a} = \int_\R \exp \left( -\frac1\Sigma a \cdot V(x) \right) \dd x.
\end{equation}
The generator $\mathcal L_a$ of \eqref{eq:SDE_H_a} is defined for all $u \in C^2(\R)$ as
\begin{equation} \label{eq:generator}
\mathcal L_a u(x) = - a \cdot V'(x) u'(x) + \Sigma u''(x),
\end{equation}
where the subscript denotes the dependence of the generator on the unknown drift coefficient $a$. From the well-known spectral theory of diffusion processes and under our assumptions on the potential $V$ we deduce that $\mathcal L_a$ has a countable set of eigenvalues (see e.g. \cite{HST98}). In particular, let $\{(\lambda_j(a),\phi_j(\cdot;a))\}_{j=0}^\infty$ be the sequence of eigenvalue-eigenfunction couples of the generator which solve the eigenvalue problem
\begin{equation} \label{eq:eigen_problem_general}
\mathcal L_a \phi_j(x;a) = -\lambda_j(a) \phi_j(x;a),
\end{equation}
which, due to \eqref{eq:generator}, is equivalent to
\begin{equation} \label{eq:eigen_problem_equation}
\Sigma \phi_j''(x;a) - a \cdot V'(x) \phi_j'(x;a) + \lambda_j(a) \phi_j(x;a) = 0,
\end{equation}
and where the eigenvalues satisfy $0=\lambda_0(a)<\lambda_1(a)<\dots<\lambda_j(a)\uparrow\infty$ and the eigenfunctions form an orthonormal basis for the weighted space $L^2(\varphi_a^0)$. We mention in passing that, by making a unitary transformation, the eigenvalue problem for the generator of the Langevin dynamics can be transformed to the standard Sturm-Liouville problem for Schrödinger operators \cite[Chapter 4]{Pav14}. We now state a formula, which has been proved in \cite{KeS99} and will be fundamental in the rest of the paper
\begin{equation} \label{eq:martingale_formula}
\E \left[ \phi_j(X_{t_n}(a);a) | X_{t_{n-1}}(a) = x \right] = e^{-\lambda_j(a)\Delta} \phi_j(x;a),
\end{equation}
where $\Delta = t_n - t_{n-1}$ is the constant distance between two consecutive observations. We now discuss how this eigenvalue problem can be used for parameter estimation. Let $J$ be a positive integer and let $\{ \beta_j(\cdot;a) \}_{j=1}^J$ be $J$ arbitrary functions $\beta_j(\cdot;a) \colon \R \to \R^M$ possibly dependent on the parameter $a$, which satisfy Assumption \ref{ass:beta_functions}(i)(ii) stated below, and define for $x,y,z \in \R$ the martingale estimating function
\begin{equation} \label{eq:def_g}
g_j(x,y,z;a) = \beta_j(z;a) \left( \phi_j(y;a) - e^{-\lambda_j(a)\Delta} \phi_j(x;a) \right).
\end{equation}
Then, given a set of observations $\{ \widetilde X_n^\epl \}_{n=0}^N$, we consider the score function $\widehat G_{N,J}^\epl \colon \mathcal A \to \R^M$ defined by
\begin{equation} \label{eq:score_function_NOfilter}
\widehat G_{N,J}^\epl(a) = \frac1\Delta \sum_{n=0}^{N-1} \sum_{j=1}^J g_j(\widetilde X_n^\epl, \widetilde X_{n+1}^\epl, \widetilde X_n^\epl; a).
\end{equation}
This function can be seen as an approximation in terms of eigenfunctions of the true score function, i.e., the gradient of the log-likelihood function with respect to the unknown parameter. The full derivation of a martingale estimating function as an approximation of the true score function is given in detail in \cite[Section 2]{BiS95}. The first step is a discretization of the gradient of the continuous-time log-likelihood, which yields a biased estimating function. Hence, the next step is adjusting this function by adding its compensator in order to obtain a zero-mean martingale. Moreover, by using the eigenfunctions of the generator, it is shown in~\cite{KeS99} that this approach is suitable for scalar diffusion processes with no multiscale structure, i.e., processes with a single characteristic length/time scale. In fact, by a classical result for ergodic diffusion processes \cite[Section 4.7]{Pav14}, any function in the $L^2$ space weighted by the invariant measure can be written as an infinite linear combination of the eigenfunctions of the generator of the diffusion process.

\begin{remark}
In the construction of the martingale estimating function $\widehat G^\epl_{N,J}(a)$ we omitted the first index $j = 0$ because, for ergodic diffusion processes, the first eigenvalue is zero, $\lambda_0(a) = 0$, and its corresponding eigenfunction is constant, $\phi_0(a) = 1$, and hence they would give $g_0(x,y,z;a) = 0$ independently of the function $\beta_0(z;a)$. Therefore, it would not provide us with any information about the unknown parameters in the drift.
\end{remark}

\paragraph{The estimator $\widehat A^\epl_{N,J}$.} The first estimator we propose for the homogenized drift coefficient $A$ is given by the solution $\widehat A^\epl_{N,J}$ of the $M$-dimensional nonlinear system
\begin{equation} \label{eq:system2solve}
\widehat G^\epl_{N,J}(a) = 0.
\end{equation}
An intuition on why $\widehat G^\epl_{N,J}$ is a good score function is given by the following result. Let $\widehat G^0_{N,J}$ be the score function where the observations of the slow variable of the multiscale process are replaced by the homogenized ones, then due to equation \eqref{eq:martingale_formula}
\begin{equation}
\E \left[ \widehat G^0_{N,J}(A) \right] = 0,
\end{equation}
which means that the zero of the expectation of the score function with homogenized observations is exactly the drift coefficient of the effective equation. In Algorithm \ref{alg:NOfilter} we summarize the main steps for computing the estimator $\widehat A^\epl_{N,J}$ and further details about the implementation can be found in Appendix \ref{app:implementation}. We finally introduce the following technical assumption which will be employed in the analysis.
\begin{assumption} \label{ass:beta_functions}
The following hold for all $a\in\mathcal A$ and for all $j=1,\dots,J$:
\begin{enumerate}
\item $\beta_j(z;a)$ is continuously differentiable with respect to $a$ for all $z\in\R$;
\item all components of $\beta_j(\cdot;a)$, $\beta_j'(\cdot;a)$, $\dot \beta_j(\cdot;a)$, $\dot \beta_j'(\cdot;a)$ are polynomially bounded;
\item the slow-scale potential $V$ is such that $\phi_j(\cdot;a)$, $\phi_j'(\cdot;a)$, $\phi_j''(\cdot;a)$, and all components of $\dot \phi_j(\cdot;a)$, $\dot \phi_j'(\cdot;a)$, $\dot \phi_j''(\cdot;a)$  are polynomially bounded;
\end{enumerate}
where the dot denotes either the Jacobian matrix or the gradient with respect to $a$.
\end{assumption}
\begin{remark}
In \cite{KeS99} the authors propose a method to choose the functions $\{ \beta_j(\cdot;a) \}_{j=1}^J$ in order to obtain optimality in the sense of \cite{GoH87}: this optimal set of functions can be seen as the projection of the score function onto the set of martingale estimating functions obtained by varying the function $\{ \beta_j(\cdot;a) \}_{j=1}^J$. For the class of diffusion processes for which the eigenfunctions are polynomials, the optimal estimating functions can be computed analytically. In fact, they are related to the moments of the transition density, which can be computed explicitly. Moreover, another procedure is to choose functions which depend only on the unknown parameter and which minimize the asymptotic variance. This approach is strongly related to the asymptotic optimality criterion considered by \cite{HeG89}. For further details on how to choose these functions we refer to \cite{KeS99}, and we remark that their calculation requires additional computational cost. Nevertheless, the theory we develop is valid for all functions which satisfy Assumptions \ref{ass:beta_functions}(i) and \ref{ass:beta_functions}(ii) and we observed in practice that choosing simple functions independent of the unknown parameter, e.g. monomials of the form $\beta_j(z;a) = z^k$ with $k \in \N$, is sufficient to obtain satisfactory estimations. We also remark that in one dimension we can characterize completely all diffusion processes whose generator has orthogonal polynomials as eigenfunctions \cite[Section 2.7]{BGL14}.  Partial results in this directions also exist in higher dimensions.
\end{remark}

\begin{algorithm}
\caption{Estimation of $A$ without filtered data} \label{alg:NOfilter}
\begin{tabbing}
\hspace*{\algorithmicindent} \textbf{Input:} \= Observations $\left\{ \widetilde X_n^\epl \right\}_{n=0}^N$. \\
\> Distance between two consecutive observations $\Delta$. \\
\> Number of eigenvalues and eigenfunctions $J$. \\
\> Functions $\left\{ \beta_j(z;a) \right\}_{j=1}^J$. \\
\> Slow-scale potential $V$. \\
\> Diffusion coefficient $\Sigma$.
\end{tabbing}
\begin{tabbing}
\hspace*{\algorithmicindent} \textbf{Output:} \= Estimation $\widehat A^\epl_{N,J}$ of $A$.
\end{tabbing}
\begin{algorithmic}[1]
\State Consider the eigenvalue problem $\Sigma \phi_j''(x;a) - a \cdot V'(x) \phi_j'(x;a) + \lambda_j(a) \phi_j(x;a) = 0$.
\State Compute the first $J$ eigenvalues $\left\{ \lambda_j(a) \right\}_{j=1}^J$ and eigenfunctions $\left\{ \phi_j(\cdot;a) \right\}_{j=1}^J$.
\State Construct the function $g_j(x,y,z;a) = \beta_j(z;a) \left( \phi_j(y;a) - e^{-\lambda_j(a)\Delta} \phi_j(x;a) \right)$.
\State Construct the score function $\widehat G_{N,J}^\epl(a) = \frac1\Delta \sum_{n=0}^{N-1} \sum_{j=1}^J g_j(\widetilde X_n^\epl, \widetilde X_{n+1}^\epl, \widetilde X_n^\epl; a)$.
\State Let $\widehat A^\epl_{N,J}$ be the solution of the nonlinear system $\widehat G_{N,J}^\epl(a) = 0$.
\end{algorithmic}
\end{algorithm}

\subsection{The filtering approach}

We now go back to our multiscale SDE \eqref{eq:SDE_MS} and, inspired by \cite{AGP20}, we propose a second estimator for the homogenized drift coefficient by filtering the data. In particular, we modify $\widehat A_{N,J}^\epl$ by filtering the observations and inserting the new data into the score function $\widehat G_{N,J}^\epl$ in order to take into account the case when the step size $\Delta$ is too small with respect to the multiscale parameter $\epl$. Let us consider the exponential kernel $k \colon \R^+ \to \R$ defined as
\begin{equation} \label{eq:def_filter_beta}
k(r) = e^{-r},
\end{equation}
for which a rigorous theory has been developed in \cite{AGP20}. We remark that this exponential kernel is a low-pass filter, which cuts the high frequencies and highlights the slowest components. We then define the filtered observations $\{ \widetilde Z_n^\epl \}_{n=0}^N$ choosing $\widetilde Z_0^\epl = 0$ and computing the weighted average for all $n = 1,\dots,N$
\begin{equation} \label{eq:Z_tilde}
\widetilde Z^\epl_n = \Delta \sum_{k=0}^{n-1} k(\Delta(n-k)) \widetilde X^\epl_k,
\end{equation}
where the fast-scale component of the original multiscale trajectory is eliminated, and we define the new score function as a modification of \eqref{eq:score_function_NOfilter}, i.e.,
\begin{equation} \label{eq:score_function_YESfilter}
\widetilde G_{N,J}^\epl(a) = \frac1\Delta \sum_{n=0}^{N-1} \sum_{j=1}^J g_j(\widetilde X_n^\epl, \widetilde X_{n+1}^\epl, \widetilde Z_n^\epl; a).
\end{equation}

\begin{remark}
Notice that the filtered data only partially replace the original data in the definition of the score function. This idea is inspired by \cite{AGP20} where the same approach is used with the maximum likelihood estimator. The importance of keeping also the original observations becomes apparent in the proofs of the main results. However, a simple intuition is provided by equation \eqref{eq:martingale_formula}. This equation is essential in order to obtain the unbiasedness of the estimators when the sampling rate $\Delta$ is independent of the multiscale parameter $\epl$, but it is not valid for the filtered process. 
\end{remark}

\paragraph{The estimator $\widetilde A^\epl_{N,J}$.} The second estimator $\widetilde A^\epl_{N,J}$ is given by the solution of the $M$-dimensional nonlinear system
\begin{equation} \label{eq:system2solve_filter}
\widetilde G^\epl_{N,J}(a) = 0.
\end{equation}
The main steps to compute the estimator $\widetilde A^\epl_{N,J}$ are highlighted in Algorithm \ref{alg:YESfilter} and additional details about the implementation can be found in Appendix \ref{app:implementation}. Note that \eqref{eq:Z_tilde} can be rewritten as
\begin{equation} \label{eq:Z_tilde_1}
\widetilde Z^\epl_n = \Delta \sum_{k=0}^{n-1} e^{-\Delta (n - k)} \widetilde X^\epl_k.
\end{equation}
We introduce its continuous version $Z_t^\epl$ which will be employed in the analysis
\begin{equation} \label{eq:Z}
Z^\epl_t = \int_0^t e^{-(t-s)} X^\epl_s \dd s.
\end{equation}
We remark that the joint process $(X_t^\epl,Z_t^\epl)$ satisfies the system of multiscale SDEs 
\begin{equation} \label{eq:systemSDE_MS}
\begin{aligned}
\d X_t^\epl &= -\alpha \cdot V'(X_t^\epl) \dd t - \frac1\epl p'\left(\frac{X_t^\epl}\epl\right) \dd t + \sqrt{2\sigma} \dd W_t, \\
\d Z^\epl_t &= \left ( X^\epl_t - Z^\epl_t \right ) \dd t,
\end{aligned}
\end{equation}
and, using the theory of homogenization, when $\epl$ goes to zero it converges in law as a random variable in $\mathcal C^0([0,T];\R^2)$ to the two-dimensional process $(X_t^0,Z_t^0)$, which solves
\begin{equation} \label{eq:systemSDE_H}
\begin{aligned}
\d X_t^0 &= -A \cdot V'(X_t^0) \dd t + \sqrt{2\Sigma} \dd W_t, \\
\d Z^0_t &= \left ( X^0_t - Z^0_t \right ) \dd t.
\end{aligned}
\end{equation}
Moreover, it has been proved in \cite{AGP20} that the two-dimensional processes $(X_t^\epl,Z_t^\epl)$ and $(X_t^0,Z_t^0)$ are geometrically ergodic and their respective invariant measures have densities with respect to the Lebesgue measure denoted respectively by $\rho^\epl = \rho^\epl(x,z)$ and $\rho^0 = \rho^0(x,z)$. Let us finally remark that given discrete observations $\widetilde X_n^\epl$ we can only compute $\widetilde Z^\epl_n$, but the theory, which has to be employed for proving the convergence results, has been studied for the continuous-time process $Z^\epl_t$.
\begin{remark} \label{rem:computational_cost}
The only difference in the construction of the estimators $\widehat A^\epl_{N,J}$ and $\widetilde A^\epl_{N,J}$ is the fact that the latter requires filtered data, which are obtained from discrete observations, and thus it is computationally more expensive. Therefore, when it is possible to use the estimator without filtered data, it is preferable to employ it.
\end{remark}

\begin{algorithm}
\caption{Estimation of $A$ with filtered data} \label{alg:YESfilter}
\begin{tabbing}
\hspace*{\algorithmicindent} \textbf{Input:} \= Observations $\left\{ \widetilde X_n^\epl \right\}_{n=0}^N$. \\
\> Distance between two consecutive observations $\Delta$. \\
\> Number of eigenvalues and eigenfunctions $J$. \\
\> Functions $\left\{ \beta_j(z;a) \right\}_{j=1}^J$. \\
\> Slow-scale potential $V$. \\
\> Diffusion coefficient $\Sigma$.
\end{tabbing}
\begin{tabbing}
\hspace*{\algorithmicindent} \textbf{Output:} \= Estimation $\widetilde A^\epl_{N,J}$ of $A$.
\end{tabbing}
\begin{algorithmic}[1]
\State Consider the eigenvalue problem $\Sigma \phi_j''(x;a) - a \cdot V'(x) \phi_j'(x;a) + \lambda_j(a) \phi_j(x;a) = 0$.
\State Compute the first $J$ eigenvalues $\left\{ \lambda_j(a) \right\}_{j=1}^J$ and eigenfunctions $\left\{ \phi_j(\cdot;a) \right\}_{j=1}^J$.
\State \emph{Compute the filtered data $\left\{ \widetilde Z_n^\epl \right\}_{n=0}^N$ as $\widetilde Z_0^\epl = 0$ and $\widetilde Z^\epl_n = \Delta \sum_{k=0}^{n-1} e^{-\Delta (n - k)} \widetilde X^\epl_k$.}
\State Construct the function $g_j(x,y,z;a) = \beta_j(z;a) \left( \phi_j(y;a) - e^{-\lambda_j(a)\Delta} \phi_j(x;a) \right)$.
\State Construct the score function $\widetilde G_{N,J}^\epl(a) = \frac1\Delta \sum_{n=0}^{N-1} \sum_{j=1}^J g_j(\widetilde X_n^\epl, \widetilde X_{n+1}^\epl, \widetilde Z_n^\epl; a)$.
\State Let $\widetilde A^\epl_{N,J}$ be the solution of the nonlinear system $\widetilde G_{N,J}^\epl(a) = 0$.
\end{algorithmic}
\end{algorithm}

\section{Main results}

In this section we present the main results of this work, i.e., the asymptotic unbiasedness of the proposed estimators. We first need to introduce the following technical assumption, which is a nondegeneracy hypothesis related to the use of the implicit function theorem for the functions \eqref{eq:score_function_NOfilter} and \eqref{eq:score_function_YESfilter} in the limit as $N \to \infty$.

\begin{assumption} \label{ass:for_Dini}
Let $A$ be the homogenized drift coefficient of equation \eqref{eq:SDE_H}. Then the following hold
\begin{enumerate}
\item $\det \left( \sum_{j=1}^J \E^{\widetilde \rho^0} \left[ \left( \beta_j(\widetilde Z_0^0;A) \otimes \nabla_a X_\Delta(A) \right) \phi_j'(X_\Delta^0;A) \right] \right) \neq 0$,
\item $\det \left( \sum_{j=1}^J \E^{\varphi^0} \left[ \left( \beta_j(X_0^0;A) \otimes \nabla_a X_\Delta(A) \right) \phi_j'(X_\Delta^0;A) \right] \right) \neq 0$,
\item $\det \left( \sum_{j=1}^J \E^{\rho^0} \left[ (\beta_j(Z_0^0;A) \otimes V'(X_0^0)) \phi_j'(X_0^0;A) \right] \right) \neq 0$,
\item $\det \left( \sum_{j=1}^J \E^{\varphi^0} \left[ (\beta_j(X_0^0;A) \otimes V'(X_0^0)) \phi_j'(X_0^0;A) \right] \right) \neq 0$,
\end{enumerate}
where $\widetilde \rho^0$ is the invariant measure of the couple $(\widetilde X_n^0, \widetilde Z_n^0)$, whose existence is guaranteed by Lemma \ref{lem:ergodicity_XZ_tilde}, and $\nabla_a X_t(a)$ is the gradient of the stochastic process $X_t(a)$ in \eqref{eq:SDE_H_a} with respect to the drift coefficient $a$.
\end{assumption}

\begin{remark}
The nondegeneracy Assumption \ref{ass:for_Dini}, which is analogous to Condition 4.2(a) in \cite{KeS99}, holds true in all nonpathological examples and does not constitute an essential limitation on the range of validity of the results proved in this paper. Further details about the necessity of this assumption for the analysis of the proposed estimator will be given in Section \ref{sec:proof_main_results}.
\end{remark}

The proofs of the following two main theorems are the focus of Section \ref{sec:full_proof}.

\begin{theorem} \label{thm:unbiasedness_hat}
Let $J$ be a positive integer. Under Assumptions \ref{ass:dissipative_setting}, \ref{ass:beta_functions}, \ref{ass:for_Dini} and if $\Delta$ is independent of $\epl$ or $\Delta = \epl^\zeta$ with $\zeta\in(0,1)$, there exists $\epl_0 > 0$ such that for all $0<\epl<\epl_0$ , an estimator $\widehat A_{N,J}^\epl$ which solves the system $\widehat G_{N,J}^\epl(\widehat A_{N,J}^\epl) = 0$ exists with probability tending to one as $N\to\infty$. Moreover
\begin{equation}
\lim_{\epl\to0} \lim_{N\to\infty} \widehat A_{N,J}^\epl = A, \qquad \text{in probability},
\end{equation}
where $A$ is the homogenized drift coefficient of equation \eqref{eq:SDE_H}.
\end{theorem}

\begin{theorem} \label{thm:unbiasedness_tilde}
Let $J$ be a positive integer. Under Assumptions \ref{ass:dissipative_setting}, \ref{ass:beta_functions}, \ref{ass:for_Dini} and if $\Delta$ is independent of $\epl$ or $\Delta = \epl^\zeta$ with $\zeta>0$ and $\zeta\neq1$, $\zeta\neq2$, there exists $\epl_0 > 0$ such that for all $0<\epl<\epl_0$ an estimator $\widetilde A_{N,J}^\epl$ which solves the system $\widetilde G_{N,J}^\epl(\widetilde A_{N,J}^\epl) = 0$ exists with probability tending to one as $N\to\infty$. Moreover
\begin{equation}
\lim_{\epl\to0} \lim_{N\to\infty} \widetilde A_{N,J}^\epl = A, \qquad \text{in probability},
\end{equation}
where $A$ is the homogenized drift coefficient of equation \eqref{eq:SDE_H}.
\end{theorem}

\begin{remark}
Notice that in both Theorem \ref{thm:unbiasedness_hat} and Theorem \ref{thm:unbiasedness_tilde} the order of the limits is important and they cannot be interchanged. In fact, we first consider the large data limit, i.e., the number of observations $N$ tends to infinity, and then we let the multiscale parameter $\epl$ vanish. Moreover, in Theorem \ref{thm:unbiasedness_tilde} the values $\zeta = 1$ and $\zeta = 2$ are not allowed because of technicalities in the proof, but we observe numerically that the estimator works well also in these two particular cases.
\end{remark}

These two theorems show that both estimators based on the multiscale data from \eqref{eq:SDE_MS} converge to the homogenized drift coefficient $A$ of \eqref{eq:SDE_H}. Since the analysis is similar for the two cases, we will mainly focus on the second score function with filtered observations and at the end of each step we will state the differences with respect to the estimator without pre-processed data.

\begin{remark} \label{rem:estimation_diffusion_coefficient}
Since the main goal of this work is the estimation of the effective drift coefficient $A$, in the numerical experiments and in the following analysis we will always assume the effective diffusion coefficient $\Sigma$ to be known. Nevertheless, we remark that our methodology can be slightly modified in order to take into account the estimation of the effective diffusion coefficient too. In fact, the parameter $a$ can be replaced by the parameter $\theta = (a,s) \in \R^{M+1}$ where $a$ stands for the drift and $s$ stands for the diffusion, yielding nonlinear systems of dimension $M+1$ corresponding to \eqref{eq:system2solve} and \eqref{eq:system2solve_filter}. The proofs of the asymptotic unbiasedness of the new estimators $\widehat \theta^\epl_{N,J}$ and $\widetilde \theta^\epl_{N,J}$ can be adjusted analogously. For completeness, we provide a more detailed explanation and a numerical experiment illustrating this approach in Section \ref{sec:num_diff}.
\end{remark}

\subsection{A particular case}

Before analysing the general framework, let us consider the simple case of the Ornstein-Uhlenbeck process, i.e. let the dimension of the parameter $N=1$ and let $V(x) = x^2/2$. Then the multiscale SDE \eqref{eq:SDE_MS} becomes
\begin{equation} \label{eq:OU_MS}
\d X_t^\epl = -\alpha X_t^\epl \dd t - \frac1\epl p'\left(\frac{X_t^\epl}\epl\right) \dd t + \sqrt{2\sigma} \dd W_t,
\end{equation}
and its homogenized version is
\begin{equation} \label{eq:OU_H}
\d X_t^0 = -A X_t^0 \dd t + \sqrt{2\Sigma} \dd W_t.
\end{equation}
Letting $a\in\mathcal A$, then the eigenfunctions $\phi_j(\cdot;a)$ and the eigenvalues $\lambda_j(a)$ satisfy
\begin{equation}
\phi_j''(x;a) - \frac a\Sigma x \phi'(x) + \frac{\lambda(a)}{\Sigma} \phi(\cdot;a) = 0.
\end{equation}
The solution of the eigenvalue problem can be computed explicitly (see \cite[Section 4.4]{Pav14}); we have
\begin{equation}
\lambda_j(a) = ja,
\end{equation}
and $\phi_j(\cdot;a)$ satisfies the recurrence relation
\begin{equation}
\phi_{j+1}(x;a) = x \phi_j(x;a) - j \frac \Sigma a \phi_{j-1}(x;a),
\end{equation}
with $\phi_0(x;a) = 1$ and $\phi_1(x;a) = x$. It is also possible to prove by induction that
\begin{equation}
\phi_j'(x;a) = j \phi_{j-1}(x).
\end{equation}
Let us consider the simplest case with only one eigenfunction, i.e. $J=1$, and $\beta_1(z;a) = z$, which implies
\begin{equation}
g_1(x,y,z;a) = z \left( y - e^{-a\Delta} x \right).
\end{equation}
Then the score functions \eqref{eq:score_function_NOfilter} and \eqref{eq:score_function_YESfilter} become
\begin{equation}
\begin{aligned}
&\widehat G_{N,1}^\epl(a) = \frac1\Delta \sum_{n=0}^{N-1} \widetilde X_n^\epl \left( \widetilde X_{n+1}^\epl - e^{-a\Delta} \widetilde X_n^\epl \right), \\
&\widetilde G_{N,1}^\epl(a) = \frac1\Delta \sum_{n=0}^{N-1} \widetilde Z_n^\epl \left( \widetilde X_{n+1}^\epl - e^{-a\Delta} \widetilde X_n^\epl \right).
\end{aligned}
\end{equation}
The solutions of the equations $\widehat G_{N,1}^\epl(a) = 0$ and $\widetilde G_{N,1}^\epl(a) = 0$ can be computed analytically and are given by
\begin{equation} \label{eq:def_Ahat_OU}
\widehat A_{N,1}^\epl = - \frac1\Delta \log \left( \frac{\sum_{n=0}^{N-1} \widetilde X_n^\epl \widetilde X_{n+1}^\epl}{\sum_{n=0}^{N-1} (\widetilde X_n^\epl)^2 } \right),
\end{equation} 
and
\begin{equation} \label{eq:def_Atilde_OU}
\widetilde A_{N,1}^\epl = - \frac1\Delta \log \left( \frac{\sum_{n=0}^{N-1} \widetilde Z_n^\epl \widetilde X_{n+1}^\epl}{\sum_{n=0}^{N-1} \widetilde Z_n^\epl \widetilde X_n^\epl } \right).
\end{equation}
Comparing these estimators with the discrete MLE defined in \cite{PaS07} without filtered data as
\begin{equation}
\widehat{\mathrm{MLE}}_{N,\Delta}^\epl = - \frac{\sum_{n=0}^{N-1} \widetilde X_n^\epl ( \widetilde X_{n+1}^\epl - \widetilde X_n^\epl )}{\Delta \sum_{n=0}^{N-1} (\widetilde X_n^\epl)^2},
\end{equation}
and the discrete MLE with filtered data
\begin{equation}
\widetilde{\mathrm{MLE}}_{N,\Delta}^\epl = - \frac{\sum_{n=0}^{N-1} \widetilde Z_n^\epl ( \widetilde X_{n+1}^\epl - \widetilde X_n^\epl )}{\Delta \sum_{n=0}^{N-1} \widetilde Z_n^\epl \widetilde X_n^\epl},
\end{equation}
we notice that they coincide in the limit as $\Delta$ vanishes. We remark that we are comparing our estimator with the discrete MLE instead of the analytical formula for the MLE in continuous time since we assume that we are observing our process at discrete times. Therefore, the continuous time MLE has to be approximated using the available discrete data \cite[Section 5.3]{Pav14}. In the following theorems we show the asymptotic limit of the estimators. We do not provide a proof for these results since Theorem \ref{thm:unbiasedness_hat_OU} and Theorem \ref{thm:unbiasedness_tilde_OU} are particular cases of Theorem \ref{thm:unbiasedness_hat} and Theorem \ref{thm:unbiasedness_tilde} respectively, and Theorem \ref{thm:biasedness_hat_OU} follows from the proof of Theorem \ref{thm:unbiasedness_hat} as highlighted in Remark \ref{rem:biasedness}.

\begin{theorem} \label{thm:unbiasedness_hat_OU}
Let $\Delta$ be independent of $\epl$ or $\Delta = \epl^\zeta$ with $\zeta\in(0,1)$. Then, under Assumption \ref{ass:dissipative_setting}, the estimator \eqref{eq:def_Ahat_OU} satisfies
\begin{equation}
\lim_{\epl \to 0} \lim_{N\to\infty} \widehat A^\epl_{N,1} = A, \quad \text{in probability},
\end{equation}
where $A$ is the drift coefficient of the homogenized equation \eqref{eq:SDE_H}.
\end{theorem}

\begin{theorem} \label{thm:biasedness_hat_OU}
Let $\Delta$ be independent of $\epl$ or $\Delta = \epl^\zeta$ with $\zeta>2$. Then, under Assumption \ref{ass:dissipative_setting}, the estimator \eqref{eq:def_Ahat_OU} satisfies
\begin{equation}
\lim_{\epl \to 0} \lim_{N\to\infty} \widehat A^\epl_{N,1} = \alpha, \quad \text{in probability},
\end{equation}
where $\alpha$ is the drift coefficient of the homogenized equation \eqref{eq:SDE_MS}.
\end{theorem}

\begin{theorem} \label{thm:unbiasedness_tilde_OU}
Let $\Delta$ be independent of $\epl$ or $\Delta = \epl^\zeta$ with $\zeta\neq1$, $\zeta\neq2$. Then, under Assumption \ref{ass:dissipative_setting}, the estimator \eqref{eq:def_Atilde_OU} satisfies
\begin{equation}
\lim_{\epl \to 0} \lim_{N\to\infty} \widetilde A^\epl_{N,1} = A, \quad \text{in probability},
\end{equation}
where $A$ is the drift coefficient of the homogenized equation \eqref{eq:SDE_H}.
\end{theorem}

\begin{remark} \label{rem:as_probability}
Notice that it is possible to write different proofs for Theorems \ref{thm:unbiasedness_hat_OU}, \ref{thm:biasedness_hat_OU} and \ref{thm:unbiasedness_tilde_OU}, which take into account the specific form of the estimators, and thus show stronger results. In fact, if the distance $\Delta$ between two consecutive observations is independent of the multiscale parameter $\epl$, then the convergences in the statements do not only hold in probability, but also almost surely. We expect that almost sure convergence can be proved for a larger class of equations, but are neither aware of related literature showing such a stronger result, nor have been able to prove it.
\end{remark}

\section{Numerical experiments}

In this section we present numerical experiments which confirm our theoretical results and show the power of the martingale estimating functions based on eigenfunctions and filtered data to correct the unbiasedness caused by discretization and the fact that we are using multiscale data to fit homogenized models. Moreover, we present a sensitivity analysis with respect to the number $N$ of observations and the number $J$ of eigenvalues and eigenfunctions taken into account. In the experiments that we present data are generated employing the Euler--Maruyama method with a fine time step $h$, in particular we set $h = \epl^3$. Letting $\Delta,T>0$, we generate data $X_t^\epl$ for $0 \le t \le T$ and we select a sequence of observations $\{ \widetilde X^\epl_n \}_{n=0}^N$, where $N=T/\Delta$ and $\widetilde X^\epl_n = X^\epl_{t_n}$ with $t_n = n\Delta$. In view of Remark \ref{rem:initial_condition} we do not require stationarity of the multiscale dynamics, hence we always set the initial condition to be $X_0^\epl = 0$. Notice that the time step $h$ is only used to generate numerically the original data and has to be chosen sufficiently small in order to have a reliable approximation of the continuous path. However, the distance between two consecutive observations $\Delta$ is the rate at which we sample the data, which we assume to know, from the original trajectory. In order to compute the filtered data $\{ \widetilde Z^\epl_n \}_{n=1}^N$ we employ equation \eqref{eq:Z_tilde_1}. We repeat this procedure for $M = 15$ different realizations of Brownian motion and we plot the average of the drift coefficients computed by the estimators. We finally remark that in order to compute our estimators we need the value of the diffusion coefficient $\Sigma$ of the homogenized equation. In all the numerical experiments we compute it exactly using the formula for the coefficient $K$ given by the theory of homogenization, but we also remark that its value could be estimated employing the subsampling technique presented in \cite{PaS07} or modifying the estimating function as explained in Remark \ref{rem:estimation_diffusion_coefficient}.

\subsection{Sensitivity analysis with respect to the number of observations} \label{sec:sensitivity}

We consider the multiscale Ornstein--Uhlenbeck process, i.e. equation \eqref{eq:SDE_MS} with $V(x)=x^2/2$, and we take $p(y)=\cos(y)$, the multiscale parameter $\epl=0.1$, the drift coefficient $\alpha = 1$ and the diffusion coefficient $\sigma = 1$. Notice that for this choice of the slow-scale potential the technical assumptions required in the main Theorems \ref{thm:unbiasedness_hat}, \ref{thm:unbiasedness_tilde} can be easily checked. We plot the results computed by the estimator $\widetilde A_{N,J}^\epl$ with $J=1$ and $\beta(x;a) = x$ and we then divide the analysis in two cases: $\Delta$ ``small'' and $\Delta$ ``big''.

\begin{figure}[t]
\centering
\includegraphics[]{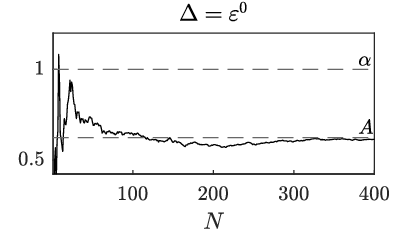}
\includegraphics[]{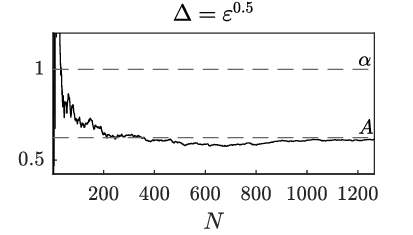} \\
\includegraphics[]{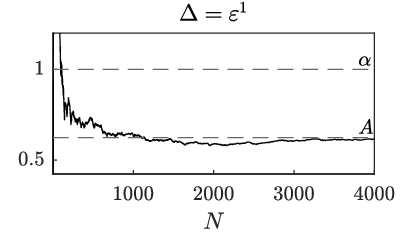}
\includegraphics[]{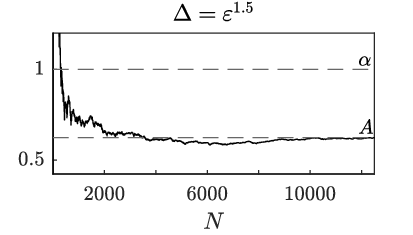}
\caption{Sensitivity analysis with respect to the number $N$ of observations for different values of $\Delta \le 1$, for the estimator $\widetilde A^\epl_{N,J}$ with $J=1$.}
\label{fig:comparisonN}
\end{figure}

Let us first consider $\Delta$ ``small'', i.e. $\Delta = \epl^\zeta$ with $\zeta = 0, 0.5, 1, 1.5$, and take $T = 400$. In Figure \ref{fig:comparisonN} we plot the results of the estimator as a function of the number of observations $N$. We remark that in this case the number of observations needed to reach convergence is strongly dependent and inversely proportional to the distance $\Delta$ between two consecutive observations. This means that in order to reach convergence we need the final time $T$ to be sufficiently large independently of $\Delta$. In fact, when the distance $\Delta$ is small, the discrete observations are a good approximation of the continuous trajectory and therefore what matters most is the length $T$ of the original path rather than the number $N$ of observations.

\begin{figure}[t]
\centering
\includegraphics[]{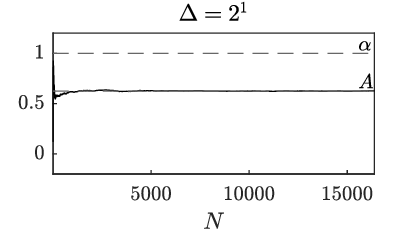}
\includegraphics[]{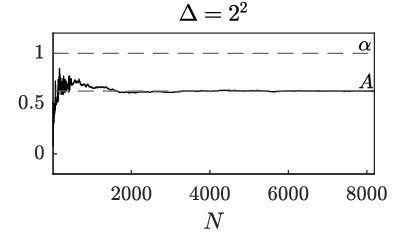} \\
\includegraphics[]{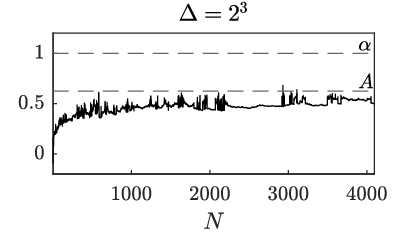}
\includegraphics[]{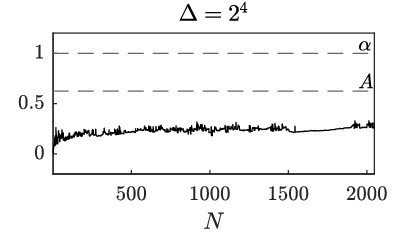}
\caption{Sensitivity analysis with respect to the number $N$ of observations for different values of $\Delta > 1$, for the estimator $\widetilde A^\epl_{N,J}$ with $J=1$.}
\label{fig:comparisonNfewData}
\end{figure}

In order to study the case $\Delta$ ``big'', i.e. $\Delta > 1$, we set $\Delta = 2^\zeta$ with $\zeta = 1, 2, 3, 4$, and take $T = 2^{15}$. Figure \ref{fig:comparisonNfewData} shows that in this case the number of observations needed to reach convergence is an increasing function of $\Delta$. Therefore, in order to have a reliable approximation of the drift coefficient of the homogenized equation, the final time $T$ has to be chosen depending on $\Delta$. This is justified by the fact that, differently from the previous case, the discrete data are less correlated and therefore they do not well approximate the continuous trajectory. In particular, when the distance $\Delta$ between two consecutive observations is very large, then in practice we need a huge amount of data because a good approximation of the unknown coefficient is obtained only if the final time $T$ is very large.

\subsection{Sensitivity analysis with respect to the number of eigenvalues and eigenfunctions}

Let us now consider equation \eqref{eq:SDE_MS} with four different slow-scale potentials
\begin{equation} \label{eq:4potentials}
V_1(x) = \frac{x^2}{2}, \qquad V_2(x) = \frac{x^4}{4}, \qquad V_3(x) = \frac{x^6}{6}, \qquad V_4(x) = \frac{x^4}{4}-\frac{x^2}{2}.
\end{equation}
The other functions and parameters of the SDE are chosen as in the previous subsection, i.e. $p(y)=\cos(y)$, $\alpha=1$, $\sigma=1$ and $\epl=0.1$. Moreover, we set $\Delta=\epl$ and $T=500$ and we vary $J = 1, \dots, 10$. The functions $\{ \beta_j \}_{j=1}^{10}$ appearing in the estimating function are given by $\beta_j(x;a) = x$ for all $j = 1, \dots, J$.

\begin{figure}[t]
\centering
\includegraphics[]{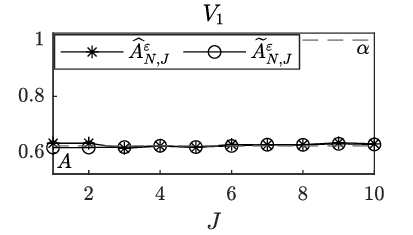}
\includegraphics[]{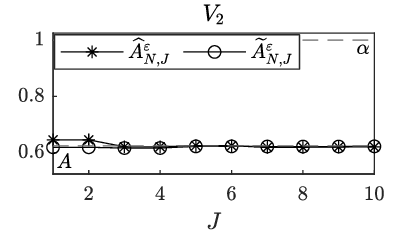} \\
\includegraphics[]{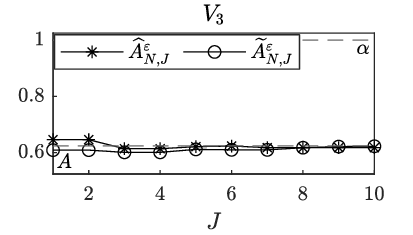}
\includegraphics[]{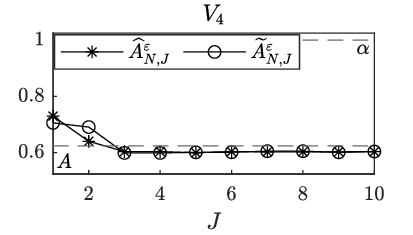}
\caption{Sensitivity analysis with respect to the number $J$ of eigenvalues and eigenfunctions for different slow-scale potentials, for the estimators $\widehat A^\epl_{N,J}$ and $\widetilde A^\epl_{N,J}$.}
\label{fig:comparisonJ}
\end{figure}

In Figure \ref{fig:comparisonJ}, where we plot the values computed by $\widehat A_{N,J}^\epl$ and $\widetilde A_{N,J}^\epl$, we observe that the number $J$ of eigenvalues and eigenfunctions slightly improve the results, in particular for the fourth potential, but the estimation stabilizes when the number of eigenvalues $J$ is still small, e.g. $J = 3$. Therefore, in order to reduce the computational cost, it seems to be preferable not to take large values of $J$. This is related to how quickly the eigenvalues grow and, therefore, how quickly the corresponding exponential terms decay. The rigorous study of the accuracy of the spectral estimators as a function of the number of eigenvalues and eigenfunctions that we take into account will be investigated elsewhere.

\subsection{Verification of the theoretical results}

We consider the same setting as in the previous subsection, i.e. equation \eqref{eq:SDE_MS} with slow-scale potentials given by \eqref{eq:4potentials} and $p(y)=\cos(y)$, $\alpha=1$, $\sigma=1$ and $\epl=0.1$. Moreover, we set $J=1$, $\beta(x;a) = x$ and $T=500$ and we choose the distance between two successive observations to be $\Delta = \epl^\zeta$ with $\zeta=0,0.1,0.2,\dots,2.5$.

\begin{figure}[t]
\centering
\includegraphics[]{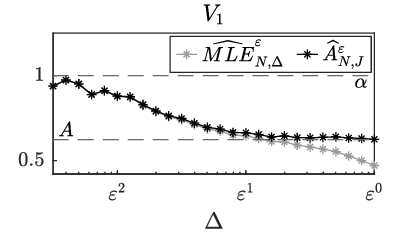}
\includegraphics[]{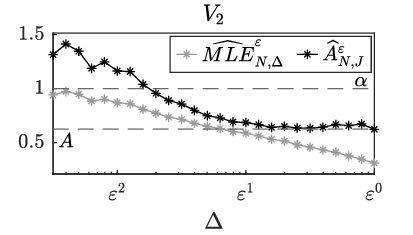}
\includegraphics[]{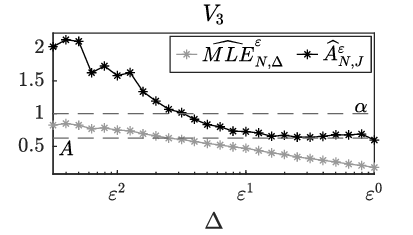}
\includegraphics[]{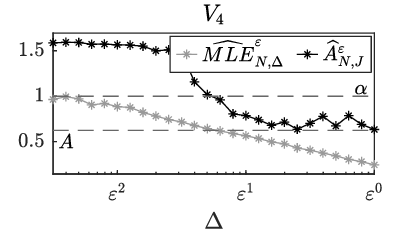}
\caption{Comparison between the discrete maximum likelihood estimator $\widehat{\mathrm{MLE}}_{N,\Delta}^\epl$ presented in \cite{PaS07} and our estimator $\widehat A_{N,J}^\epl$ with $J=1$ without filtered data as a function of the distance $\Delta$ between two successive observations for different slow-scale potentials.}
\label{fig:comparisonDelta_big}
\end{figure}

In Figure \ref{fig:comparisonDelta_big} we compare our martingale estimator $\widehat A_{N,J}^\epl$ without filtered data with the discrete maximum likelihood estimator denoted $\widehat{\mathrm{MLE}}_{N,\Delta}^\epl$. The MLE does not provide good results for two reasons:
\begin{itemize}
\item if $\Delta$ is small, more precisely if $\Delta = \epl^\zeta$ with $\zeta > 1$, sampling the data does not completely eliminate the fast-scale components of the original trajectory, therefore, since we are employing data generated by the multiscale model, the estimator is trying to approximate the drift coefficient $\alpha$ of the multiscale equation, rather than the one of the homogenized equation;
\item if $\Delta$ is relatively big, in particular if $\Delta = \epl^\zeta$ with $\zeta \in [0,1)$, then we are taking into account only the slow-scale components of the original trajectory, but a bias is still introduced because we are discretizing an estimator which is usually used for continuous data.
\end{itemize}
Nevertheless, as observed in these numerical experiments and investigated in greater detail in \cite{PaS07}, there exists an optimal value of $\Delta$ such that $\widehat{\mathrm{MLE}}_{N,\Delta}^\epl$ works well, but this value is not known a priori and is strongly dependent on the problem, hence this technique is not robust. Figure \ref{fig:comparisonDelta_big} shows that the second issue, i.e., when $\Delta$ is relatively big, can be solved employing $\widehat A_{N,J}^\epl$, an estimator for discrete observations, and that filtering the data is not needed as proved in Theorem \ref{thm:unbiasedness_hat}.

\begin{figure}[t]
\centering
\includegraphics[]{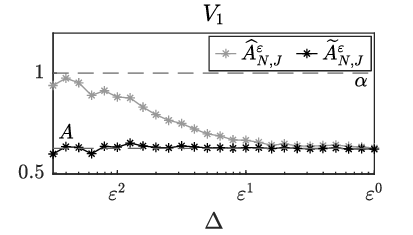}
\includegraphics[]{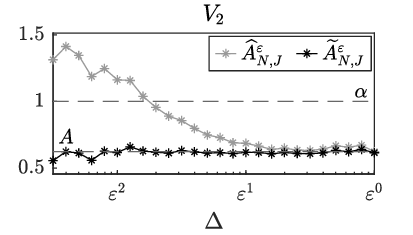}
\includegraphics[]{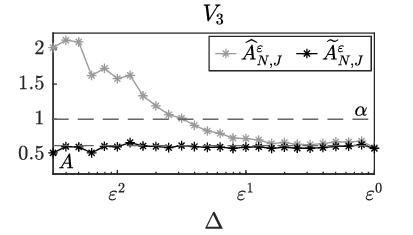}
\includegraphics[]{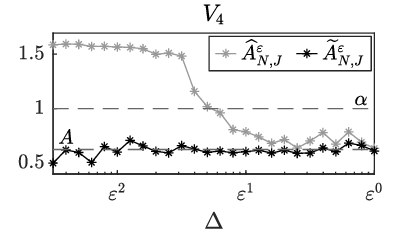}
\caption{Comparison between our two estimators $\widehat A_{N,J}^\epl$ without filtered data and $\widetilde A_{N,J}^\epl$ with filtered data with $J=1$ as a function of the distance $\Delta$ between two successive observations for different slow-scale potentials.}
\label{fig:comparisonDelta_small}
\end{figure}

Then, in order to solve also the first problem, in Figure \ref{fig:comparisonDelta_small} we compare $\widehat A_{N,J}^\epl$ with our martingale estimator $\widetilde A_{N,J}^\epl$ with filtered data. We observe that inserting filtered data in the estimator allows us to disregard the fast-scale components of the original trajectory and to obtain good approximations of the drift coefficient $A$ of the homogenized equation independently of $\Delta$, as already shown in Theorem \ref{thm:unbiasedness_tilde}. In particular, we notice that the results still improve even for big values of $\Delta$ if we employ the estimator based on filtered data. Finally, as highlighted in Remark \ref{rem:biasedness}, we observe that the limiting value of the estimator $\widehat A_{N,J}^\epl$ as the number of observations $N$ goes to infinity and the multiscale parameter $\epl$ vanishes is strongly dependent on the problem and can not be computed theoretically. However, if we consider the slow-scale potential $V_1(x) = x^2/2$, i.e. the multiscale Ornstein–Uhlenbeck process, then the limit, as proved in Theorem \ref{thm:biasedness_hat_OU}, is the drift coefficient $\alpha$ of the multiscale equation.

\subsection{Multidimensional drift coefficient} \label{sec:multi_coeff}

In this experiment we consider a multidimensional drift coefficient, in particular we set $N=2$. We then consider the bistable potential, i.e.,
\begin{equation}
V (x) = \begin{pmatrix} \frac{x^4}{4} & -\frac{x^2}{2} \end{pmatrix}^\top,
\end{equation}
and the fast-scale potential $p(y) = \cos(y)$. We choose the exact drift coefficient of the multiscale equation \eqref{eq:SDE_MS} to be $\alpha = \begin{pmatrix} 1.2 & 0.7\end{pmatrix}^\top$ and the diffusion coefficient to be $\sigma = 0.7$. We also set the number of eigenfunctions $J = 1$, the function $\beta(x;a) = \begin{pmatrix} x^3 & x \end{pmatrix}^\top$, the distance between two consecutive observations $\Delta = 1$ and the final time $T = 1000$. We then compute the estimator $\widehat A^\epl_{N,J}$ after $N = 100, 200, \dots, 1000$ observations and in Figure \ref{fig:example2D} we plot the result of the experiment for the cases $\epl = 0.1$ and $\epl = 0.05$. Since we are analysing the case $\Delta$ independent of $\epl$, filtering the data is not necessary and therefore we consider the estimator $\widehat A^\epl_{N,J}$ which is computationally less expensive to compute.

\begin{figure}[t]
\centering
\includegraphics[]{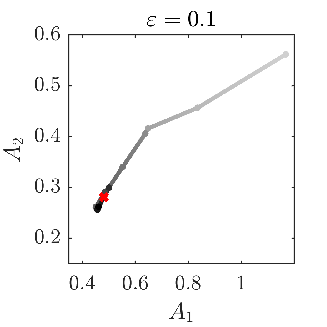}
\includegraphics[]{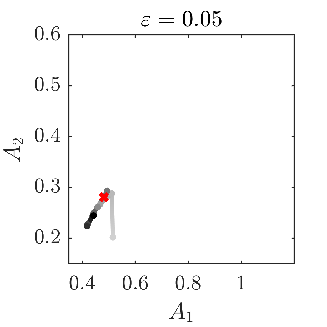} 
\hspace{0.25cm}
\includegraphics[]{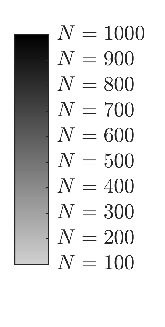}
\caption{Evolution in time of the estimator $\widehat A_{N,J}^\epl$ with $J=1$ for a two-dimensional drift coefficient.}
\label{fig:example2D}
\end{figure}

We observe that the estimation is approaching the exact value $A$ of the drift coefficient of the homogenized equation as the number of observations increases, until it starts oscillating around the true value $A = \begin{pmatrix} 0.48 & 0.28\end{pmatrix}^\top$. Moreover, we notice that the time needed to reach a neighborhood of $A$ is smaller when the multiscale parameter $\epl$ is closer to its vanishing limit. In Table \ref{tab:error2D} we report the absolute error $\widehat e^\epl_N$ defined as
\begin{equation} \label{eq:error_def}
\widehat e^\epl_N = \norm{A - \widehat A^\epl_{N,J}}_2,
\end{equation}
where $\norm{\cdot}_2$ denotes the euclidean norm, varying the number of observations $N$ for the two values of the multiscale parameter.

\begin{table}
\centering
\begin{tabular}{ccccccccccc}
\toprule
$N$ & $100$ & $200$ & $300$  & $400$ & $500$ & $600$ & $700$  & $800$ & $900$ & $1000$ \\ 
\midrule
$\epl = 0.1$ & $0.742$ & $0.395$ & $0.215$ & $0.201$ & $0.093$ & $0.036$ & $0.011$ & $0.027$ & $0.034$ & $0.028$ \\
$\epl = 0.05$ & $0.086$ & $0.031$ & $0.019$ & $0.031$ & $0.018$ & $0.049$ & $0.081$ & $0.085$ & $0.055$ & $0.053$ \\
\bottomrule
\end{tabular}
\caption{Absolute error $\widehat e^\epl_N$ defined in \eqref{eq:error_def} between the homogenized drift coefficient $A$ and the estimator $\widehat A^\epl_{N,J}$ with $J=1$ for a two-dimensional drift coefficient.}
\label{tab:error2D}
\end{table}

\subsection{Multidimensional stochastic process: interacting particles} \label{sec:IP}

In this section we consider a system of $d$ interacting particles in a two-scale potential, a problem with a wide range of applications which has been studied in \cite{GoP18}. For $t \in [0,T]$ and for all $i = 1, \dots, d$, consider the system of SDEs
\begin{equation} \label{eq:systemIP_MS}
\dd X_i^\epl(t) = - \alpha X_i^\epl(t) \dd t - \frac1\epl p' \left( \frac{X_i^\epl(t)}{\epl} \right) - \frac{\theta}{d} \sum_{j=1}^d \left( X_i^\epl(t) - X_j^\epl(t) \right) \dd t + \sqrt{2\sigma} \dd W_i(t).
\end{equation}
In this paper we fix the number of particles and study the performance of our estimators as $\epl$ vanishes. The very interesting problem of inference for mean field SDEs, obtained in the limit as $d \to \infty$, will be investigated elsewhere. It can be shown (see e.g. \cite[Section 2.1]{GoP18} and \cite{DuP16,DGP21}) that $(X_1^\epl, \dots X_d^\epl)$ converges in law as $\epl$ goes to zero to the solution $(X_1^0, \dots, X_d^0)$ of the homogenized system 
\begin{equation} \label{eq:systemIP_H}
\dd X_i^0(t) = - A X_i^0(t) \dd t  - \frac{\Theta}{d} \sum_{j=1}^d \left( X_i^0(t) - X_j^0(t) \right) \dd t + \sqrt{2\Sigma} \dd W_i(t).
\end{equation}
where $\Theta = K \theta$ and $K$ is defined in \eqref{eq:K_H}. Moreover, the first eigenvalue and eigenfunction of the generator of the homogenized system can be computed explicitly and they are given respectively by
\begin{equation}
\phi_1(x_1, \dots, x_d) = \sum_{i=1}^d x_i \qquad \text{and} \qquad \lambda_1 = A.
\end{equation}
Hence, letting $\Delta>0$ independent of $\epl$, given a sequence of observations $( (\widetilde X_1^\epl)_n, \dots (\widetilde X_d^\epl)_n)_{n=0}^N$, we can express the estimators analytically
\begin{equation}
\begin{aligned}
\widehat A^\epl_{N,1} &= - \frac1\Delta \log \left( \frac{\sum_{n=0}^{N-1} \left( \sum_{i=1}^d (\widetilde X_i^\epl)_n \right) \left( \sum_{i=1}^d (\widetilde X_i^\epl)_{n+1} \right)} {\sum_{n=0}^{N-1} \left( \sum_{i=1}^d (\widetilde X_i^\epl)_n \right)^2} \right), \\
\widetilde A^\epl_{N,1} &= - \frac1\Delta \log \left( \frac{\sum_{n=0}^{N-1} \left( \sum_{i=1}^d (\widetilde Z_i^\epl)_n \right) \left( \sum_{i=1}^d (\widetilde X_i^\epl)_{n+1} \right)} {\sum_{n=0}^{N-1} \left( \sum_{i=1}^d (\widetilde Z_i^\epl)_n \right) \left( \sum_{i=1}^d (\widetilde X_i^\epl)_n \right)} \right).
\end{aligned}
\end{equation}
Let us now set $p(y) = \cos(y)$, $\alpha = 1$, $\sigma = 1$ and $\theta = 1$. We then simulate system \eqref{eq:systemIP_MS} for different final times $T = 100, 200, \dots, 1000$ and approximate the drift coefficient $A$ of the homogenized system \eqref{eq:systemIP_H} for $d = 2$ and $d = 5$. In Figure \ref{fig:exampleIP} and Figure \ref{fig:exampleIP_filter} we plot the results respectively of the estimators $\widehat A^\epl_{N,J}$ with $\Delta = 1$ and $\widetilde A^\epl_{N,J}$ with $\Delta = \epl$ for two different values of $\epl = 0.1, 0.05$. As expected, we observe that our estimator provides a better approximation of the unknown coefficient $A$ when the time $T$ increases and that this value stabilizes after approximately $T = 500$.

\begin{figure}[t]
\centering
\includegraphics[]{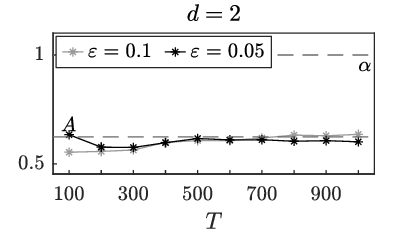}
\includegraphics[]{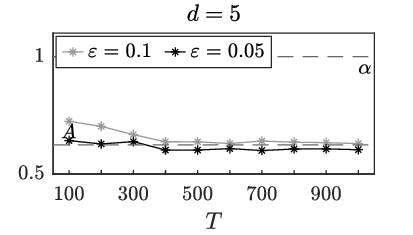}
\caption{Evolution in time of the estimator $\widehat A_{N,J}^\epl$ with $J = 1$ for a $d$-dimensional system of interacting particles with sampling rate $\Delta = 1$.}
\label{fig:exampleIP}
\end{figure}

\begin{figure}[t]
\centering
\includegraphics[]{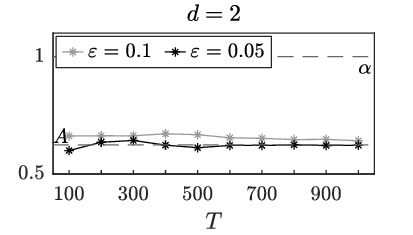}
\includegraphics[]{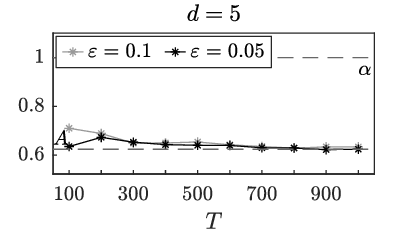}
\caption{Evolution in time of the estimator $\widetilde A_{N,J}^\epl$ with $J = 1$ for a $d$-dimensional system of interacting particles with sampling rate $\Delta = \epl$.}
\label{fig:exampleIP_filter}
\end{figure}

\subsection{Simultaneous inference of drift and diffusion coefficients} \label{sec:num_diff}
	
As highlighted by Remark \ref{rem:estimation_diffusion_coefficient}, a small modification of our methodology allows us to estimate the diffusion coefficient, in addition to drift coefficients. Define the parameter $\theta = \begin{pmatrix} a^\top & s \end{pmatrix}^\top \in \R^{M+1}$, whose exact value is given by $\theta_0 = \begin{pmatrix} A^\top & \Sigma \end{pmatrix}^\top \in \R^{M+1}$, where $A$ and $\Sigma$ are the drift and diffusion coefficients of the homogenized equation, respectively. Then, the eigenvalue problem reads for all $j \in \N$
\begin{equation}
s \phi_j''(x;\theta) - a \cdot V'(x) \phi_j'(x;\theta) + \lambda_j(\theta) \phi_j(x;\theta) = 0,
\end{equation}
where the eigenvalues and eigenfunctions are now dependent on the new parameter $\theta$. Accordingly, also the functions $\{ \beta_j \}_{j=1}^J$ can be chosen dependent on both the drift and diffusion coefficients and, moreover, they have to take values in $\R^{M+1}$, i.e., $\beta_j(\cdot;\theta) \colon \R \to \R^{M+1}$. Therefore, the new score functions $\widehat G^\epl_{N,J}$ and $\widetilde G^\epl_{N,J}$ are defined from $\Theta = \mathcal A \times \mathcal S \subset \R^{M+1}$, which is the set of admissible parameters $\theta$, to $\R^{M+1}$ and thus give nonlinear systems of dimension $M+1$. Finally, the solutions $\widehat \theta_{N,J}^\epl$ and $\widetilde \theta_{N,J}^\epl$ of the systems are the estimators of both the drift and diffusion coefficients of the homogenized equation. In fact, small modifications in the proofs of the main results, in particular in the notation, yield the asymptotic unbiasedness of the estimators under the same conditions, i.e.,
\begin{equation}
\lim_{\epl \to 0} \lim_{N \to \infty} \widehat \theta_{N,J}^\epl = \lim_{\epl \to 0} \lim_{N \to \infty} \widetilde \theta_{N,J}^\epl = \theta_0 = \begin{pmatrix} A^\top & \Sigma \end{pmatrix}^\top, \qquad \text{in probability}.
\end{equation}
Consider now the same setting of Section \ref{sec:sensitivity}, i.e., the multiscale Ornstein-Uhlebeck potential with $V(x) = x^2/2$, $p(y) = \cos(y)$, $\alpha = 1$, $\sigma = 1$ and let us assume that both the drift and diffusion coefficients are unknown. We remark that in this case we have $M = 1$. Then, set the final time $T = 1000$, the sampling rate $\Delta = 1$ and the number of eigenfunctions and eigenvalues $J = 2$. Moreover, we choose the functions $\beta_1(x;\theta) = \beta_2(x;\theta) = \begin{pmatrix} x^2 & x \end{pmatrix}^\top$. Since the distance between two consecutive observations is independent of the multiscale parameter $\epl$, we consider the estimator $\widehat A^\epl_{N,J}$ without filtered data. In Figure \ref{fig:drift&diffusion} we plot the evolution of our estimator varying the number of observations $N$ for two different values of $\epl$, in particular $\epl = 0.1$ and $\epl = 0.05$. We observe that if the multiscale parameter is smaller, then the number of observations needed to obtain a reliable approximation of the unknown parameters is lower.

\begin{figure}[t]
\centering
\includegraphics[]{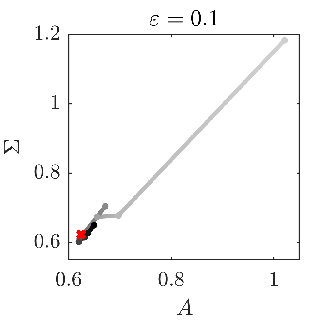}
\includegraphics[]{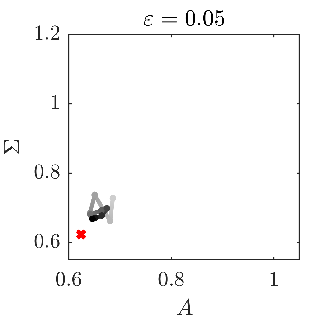}
\hspace{0.25cm}
\includegraphics[]{figures/legend_gray}
\caption{Simultaneous inference of drift and diffusion coefficient for the estimator $\widehat A^\epl_{N,J}$ with $J = 2$.}
\label{fig:drift&diffusion}
\end{figure}

\section{Asymptotic unbiasedness} \label{sec:full_proof}

In this section we prove our main results. The plan of the proof is the following:
\begin{itemize}
\item we first study the limiting behaviour of the score functions $\widehat G_{N,J}^\epl$ and $\widetilde G_{N,J}^\epl$ defined in \eqref{eq:score_function_NOfilter} and \eqref{eq:score_function_YESfilter} as the number of observations $N$ goes to infinity, i.e., as the final time $T$ tends to infinity;
\item we then show the continuity of the limit of the score functions obtained in the previous step and we compute their limits as the multiscale parameter $\epl$ vanishes (Section \ref{sec:proof_continuity_limit});
\item we finally prove our main results, i.e., the asymptotic unbiasedness of the drift estimators (Section \ref{sec:proof_main_results}).
\end{itemize}

We first define the Jacobian matrix of the function $g_j$ introduced in \eqref{eq:def_g} with respect to $a$:
\begin{equation} \label{eq:def_h}
\begin{aligned} 
h_j(x,y,z;a) &= \dot \beta_j(z;a) \left( \phi_j(y,a) - e^{-\lambda_j(a)\Delta}\phi_j(x;a) \right) \\
&\quad + \beta_j(z;a) \otimes \left( \dot \phi_j(y;a) - e^{-\lambda_j(a)\Delta} \left( \dot \phi_j(x;a) - \Delta \dot \lambda_j(a)\phi_j(x,a) \right) \right),
\end{aligned}
\end{equation}
which will be employed in the following and where $\otimes$ denotes the outer product in $\R^M$ and the dot denotes either the Jacobian matrix or the gradient with respect to $a$, e.g. $h_j = \dot g_j$. Then note that, under Assumption \ref{ass:dissipative_setting}, due to ergodicity and stationarity and by \cite[Lemma 3.1]{BiS95} we have
\begin{equation} \label{eq:def_G_hat}
\lim_{N\to\infty} \frac1N \widehat G_{N,J}^\epl(a) = \frac1\Delta \sum_{j=1}^J \E^{\varphi^\epl} \left[ g_j \left( X_0^\epl, X_\Delta^\epl, X_0^\epl; a \right) \right] \eqdef \widehat{\mathcal G}_J(\epl,a),
\end{equation}
and
\begin{equation} \label{eq:def_G_tilde}
\lim_{N\to\infty} \frac1N \widetilde G_{N,J}^\epl(a) = \frac1\Delta \sum_{j=1}^J \E^{\widetilde \rho^\epl} \left[ g_j \left( X_0^\epl, X_\Delta^\epl, \widetilde Z_0^\epl; a \right) \right] \eqdef \widetilde{\mathcal G}_J(\epl,a),
\end{equation}
where $\E^{\varphi^\epl}$ and $\E^{\widetilde \rho^\epl}$ denotes respectively that $X_0^\epl$ and $(X_0^\epl, \widetilde Z_0^\epl)$ are distributed according to their invariant distribution. We remark that the invariant distribution $\widetilde \rho^\epl$ exists due to Lemma \ref{lem:ergodicity_XZ_tilde}. By equation \eqref{eq:def_h} the Jacobian matrices of $\widehat{\mathcal G}_J(\epl,a)$ and $\widetilde{\mathcal G}_J(\epl,a)$ with respect to $a$ are given by
\begin{equation} \label{eq:def_H_hat}
\widehat{\mathcal H}_J(\epl,a) \defeq \frac{\partial}{\partial a}\widehat{\mathcal G}_J(\epl,a) = \frac1\Delta \sum_{j=1}^J \E^{\varphi^\epl} \left[ h_j \left( X_0^\epl, X_\Delta^\epl, X_0^\epl; a \right) \right],
\end{equation}
and
\begin{equation} \label{eq:def_H_tilde}
\widetilde{\mathcal H}_J(\epl,a) \defeq \frac{\partial}{\partial a}\widetilde{\mathcal G}_J(\epl,a) = \frac1\Delta \sum_{j=1}^J \E^{\widetilde \rho^\epl} \left[ h_j \left( X_0^\epl, X_\Delta^\epl, \widetilde Z_0^\epl; a \right) \right].
\end{equation}

\subsection{Continuity of the limit of the score function} \label{sec:proof_continuity_limit}

In this section, we first prove the continuity of the functions $\widehat{\mathcal G}_J, \widetilde{\mathcal G}_J \colon (0,\infty) \times \mathcal A \to \R^M$ and $\widehat{\mathcal H}_J, \widetilde{\mathcal H}_J, \colon (0,\infty) \times \mathcal A \to \R^{M \times M}$. We then study the limit of these functions for $\epl \to 0$. As the proof for the filtered and the non-filtered are similar, we will concentrate on the filtered case and comment on the non-filtered case. Before entering into the proof, we give two preliminary technical lemmas which will be used repeatedly and whose proof can be found respectively in Appendix \ref{app:technical_results} and Appendix \ref{app:approximation_formula}.

\begin{lemma} \label{lem:Ztilde_bounded_moments}
Let $\widetilde Z^\epl$ be defined in \eqref{eq:Z_tilde} and distributed according to the invariant measure $\widetilde \rho^\epl$ of the process $(\widetilde X_n, \widetilde Z_n)$. Then for any $p\ge1$ there exists a constant $C>0$ uniform in $\epl$ such that
\begin{equation}
\E^{\widetilde \rho^\epl} \abs{\widetilde Z^\epl}^p \le C.
\end{equation}  
\end{lemma}

\begin{lemma} \label{lem:expansion_fXD}
Let $f \colon \R \to \R$ be a $\mathcal C^\infty(\R)$ function which is polynomially bounded along with all its derivatives. Then
\begin{equation}
f(X_\Delta^\epl) = f(X_0^\epl) - A \cdot V'(X_0^\epl) f'(X_0^\epl) \Delta + \Sigma f''(X_0^\epl) \Delta + \sqrt{2\sigma} \int_0^\Delta f'(X_t^\epl) (1+\Phi'(Y_t^\epl)) \dd W_t + R(\epl,\Delta),
\end{equation}
where $R(\epl,\Delta)$ satisfies for all $p\ge1$ and for a constant $C>0$ independent of $\Delta$ and $\epl$
\begin{equation}
\left( \E^{\varphi^\epl} \abs{R(\epl,\Delta)}^p \right)^{1/p} \le C(\epl + \Delta^{3/2}).
\end{equation}
\end{lemma}

We start here with a continuity result for the score function and its Jacobian matrix with respect to the unknown parameter.

\begin{proposition} \label{pro:continuity_G_tilde}
Under Assumption \ref{ass:beta_functions}, the functions $\widetilde{\mathcal G}_J \colon (0,\infty) \times \mathcal A \to \R^M$ and $\widetilde{\mathcal H}_J, \colon (0,\infty) \times \mathcal A \to \R^{M \times M}$ defined in \eqref{eq:def_G_tilde} and \eqref{eq:def_H_tilde}, where $\Delta$ can be either independent of $\epl$ or $\Delta=\epl^\zeta$ with $\zeta>0$, are continuous.
\end{proposition}
\begin{proof}
We only prove the statement for $\widetilde{\mathcal G}_J$, then the argument is similar for $\widetilde{\mathcal H}_J$. Letting $\epl^*\in(0,\infty)$ and $a^*\in\mathcal A$, we want to show that
\begin{equation}
\lim_{(\epl,a) \to (\epl^*,a^*)} \norm{\widetilde{\mathcal G}_J(\epl,a) - \widetilde{\mathcal G}_J(\epl^*,a^*)} = 0.
\end{equation}
By the triangle inequality we have
\begin{equation}
\norm{\widetilde{\mathcal G}_J(\epl,a) - \widetilde{\mathcal G}_J(\epl^*,a^*)} \le \norm{\widetilde{\mathcal G}_J(\epl,a) - \widetilde{\mathcal G}_J(\epl,a^*)} + \norm{\widetilde{\mathcal G}_J(\epl,a^*) - \widetilde{\mathcal G}_J(\epl^*,a^*)} \eqdef Q_1(\epl,a) + Q_2(\epl),
\end{equation}
then we divide the proof in two steps and we show that the two terms vanish. \\
\textbf{Step 1:} $Q_1(\epl,a) \to 0$ as $(\epl,a) \to (\epl^*,a^*)$. \\
Since $\beta_j$ and $\phi_j$ are continuously differentiable with respect to $a$ for all $j=1,\dots,J$ respectively due to Assumption \ref{ass:beta_functions} and Lemma \ref{lem:regularity_eigen_a}, then also $g_j$ is continuously differentiable with respect to $a$. Therefore, by the mean value theorem for vector-valued functions we have 
\begin{equation}
\begin{aligned}
Q_1(\epl,a) &\le \frac1\Delta \sum_{j=1}^J \norm{\E^{\widetilde \rho^\epl} \left[ g_j(X_0^\epl, X_\Delta^\epl, \widetilde Z_0^\epl; a) \right] - \E^{\widetilde \rho^\epl} \left[ g_j(X_0^\epl, X_\Delta^\epl, \widetilde Z_0^\epl; a^*) \right]} \\
&= \frac1\Delta \sum_{j=1}^J \norm{\int_0^1 \E^{\widetilde \rho^\epl} \left[ h_j(X_0^\epl, X_\Delta^\epl, \widetilde Z_0^\epl; a^* + t(a-a^*)) \right] \dd t \; (a-a^*)}.
\end{aligned}
\end{equation}
Then, letting $C>0$ be a constant independent of $\epl$, since $\beta_j$ and $\phi_j$ are polynomially bounded still by Assumption \ref{ass:beta_functions} and $X_0^\epl$, $X_\Delta^\epl$ and $\widetilde Z_0^\epl$ have bounded moments of any order by \cite[Corollary 5.4]{PaS07} and Lemma \ref{lem:Ztilde_bounded_moments}, we obtain
\begin{equation}
Q_1(\epl,a) \le \frac C\Delta \norm{a-a^*},
\end{equation}
which implies that $Q_1(\epl,a)$ vanishes as $(\epl,a)$ goes to $(\epl^*,a^*)$ both if $\Delta$ is independent of $\epl$ and if $\Delta=\epl^\xi$. \\
\textbf{Step 2:} $Q_2(\epl) \to 0$ as $\epl \to \epl^*$. \\
If $\Delta$ is independent of $\epl$, then we have
\begin{equation}
\begin{aligned}
\lim_{\epl\to\epl^*} Q_2(\epl) &= \lim_{\epl\to\epl^*} \norm{\frac1\Delta \sum_{j=1}^J \E^{\widetilde \rho^\epl} \left[ g_j(X_0^\epl, X_\Delta^\epl, \widetilde Z_0^\epl; a^*) \right] - \frac1\Delta \sum_{j=1}^J \E^{\widetilde \rho^{\epl^*}} \left[ g_j(X_0^{\epl^*}, X_\Delta^{\epl^*}, \widetilde Z_0^{\epl^*}; a^*) \right]} \\
&\le \lim_{\epl\to\epl^*} \frac1\Delta \sum_{j=1}^J \norm{\E^{\widetilde \rho^\epl} \left[ g_j(X_0^\epl, X_\Delta^\epl, \widetilde Z_0^\epl; a^*) \right] - \E^{\widetilde \rho^{\epl^*}} \left[ g_j(X_0^{\epl^*}, X_\Delta^{\epl^*}, \widetilde Z_0^{\epl^*}; a^*) \right]},
\end{aligned}
\end{equation}
and the right hand side vanishes due to the continuity of $g_j$ for all $j=1,\dots,J$ and the continuity of the solution of a stochastic differential equation with respect to a parameter (see \cite[Theorem 2.8.1]{Kry09}). Let us now consider the case $\Delta = \epl^\zeta$ with $\zeta>0$ and let us assume, without loss of generality, that $\epl>\epl^*$. Denoting $\Delta^* = (\epl^*)^\zeta$ and applying Itô's lemma we have for all $j=1,\dots,J$
\begin{equation}
\begin{aligned}
\phi_j(X_\Delta^\epl;a^*) &= \phi_j(X_{\Delta^*}^{\epl};a^*) - \alpha \cdot \int_{\Delta^*}^\Delta V'(X_t^\epl) \phi_j'(X_t^\epl;a^*) \dd t - \frac1\epl \int_{\Delta^*}^\Delta \phi_j'(X_t^\epl;a^*) p' \left( \frac{X_t^\epl}{\epl} \right) \dd t \\
&\quad + \sigma \int_{\Delta^*}^\Delta \phi_j''(X_t^\epl;a^*) \dd t + \sqrt{2\sigma} \int_{\Delta^*}^\Delta \phi_j'(X_t^\epl;a^*) \dd W_t,
\end{aligned}
\end{equation}
then we can write
\begin{equation}
\widetilde {\mathcal G}_J(\epl,a^*) = \frac1\Delta \sum_{j=1}^J \left( \E^{\widetilde \rho^\epl} \left[ \beta_j(\widetilde Z_0^\epl;a^*) \phi_j(X_{\Delta^*}^\epl;a^*) \right] - e^{-\lambda(a^*)\Delta} \E^{\widetilde \rho^\epl} \left[ \beta_j(\widetilde Z_0^\epl;a^*) \phi_j(X_0^\epl;a^*) \right] \right) + R(\epl), 
\end{equation}
where $R(\epl)$ is given by
\begin{equation}
\begin{aligned}
R(\epl) &= -\frac1\Delta \sum_{j=1}^J \int_{\Delta^*}^\Delta \E^{\widetilde \rho^\epl} \left[ \beta_j(\widetilde Z_0^\epl;a^*) \phi_j'(X_t^\epl;a^*) \alpha \cdot V'(X_t^\epl) \right] \dd t \\
&\quad -\frac{1}{\epl\Delta} \sum_{j=1}^J \int_{\Delta^*}^\Delta \E^{\widetilde \rho^\epl} \left[ \beta_j(\widetilde Z_0^\epl;a^*) \phi_j'(X_t^\epl;a^*) p' \left( \frac{X_t^\epl}{\epl} \right) \right] \dd t \\
&\quad + \frac\sigma\Delta \int_{\Delta^*}^\Delta \E^{\widetilde \rho^\epl} \left[ \beta_j(\widetilde Z_0^\epl;a^*) \phi_j''(X_t^\epl;a^*) \right] \dd t + \frac{\sqrt{2\sigma}}{\Delta} \sum_{j=1}^J \E^{\widetilde \rho^\epl} \left[ \int_{\Delta^*}^\Delta \beta_j(\widetilde Z_0^\epl;a^*) \phi_j'(X_t^\epl;a^*) \dd W_t \right].
\end{aligned}
\end{equation}
Let $C>0$ be independent of $\epl$ and notice that since $p'$ is bounded, $\beta_j,\phi_j',\phi_j'',V'$ are polynomially bounded and $X_t^\epl$ and $\widetilde Z_0^\epl$ have bounded moments of any order by \cite[Corollary 5.4]{PaS07} and Lemma \ref{lem:Ztilde_bounded_moments}, applying Hölder's inequality we obtain
\begin{equation} \label{eq:bound_RDelta_star}
\abs{R(\epl)} \le \frac C\Delta \left( \norm{\alpha} + \sigma + \frac1\epl \right) (\Delta - \Delta^*) + \frac C\Delta \sqrt{2\sigma} (\Delta - \Delta^*)^{1/2}.
\end{equation} 
Therefore, by the continuity of the solution of a stochastic differential equation with respect to a parameter (see \cite{MPP10}) and due to the bound \eqref{eq:bound_RDelta_star}, we deduce that
\begin{equation}
\lim_{\epl\to\epl^*} \widetilde{\mathcal G}_J(\epl,a^*) = \frac{1}{\Delta^*} \sum_{j=1}^J \E^{\widetilde \rho^{\epl^*}} \left[ \beta_j(\widetilde Z_0^{\epl^*};a^*) \left( \phi_j(X_{\Delta^*}^{\epl^*};a^*) - e^{-\lambda(a^*)\Delta^*} \phi_j(X_0^{\epl^*};a^*) \right) \right] = \widetilde{\mathcal G}_J(\epl^*,a^*),
\end{equation}
which implies that $Q_2(\epl)$ vanishes as $\epl$ goes to $\epl^*$ and concludes the proof.
\end{proof}

\begin{remark} \label{rem:continuity_G_tilde}
Notice that the proof of Proposition \ref{pro:continuity_G_tilde} can be repeated analogously for the functions $\widehat{\mathcal G}_J \colon (0,\infty) \times \mathcal A \to \R^M$ and $\widehat{\mathcal H}_J \colon (0,\infty) \times \mathcal A \to \R^{M \times M}$ without filtered data in order to prove their continuity.
\end{remark}

Next we study the limit as $\epl$ vanishes and we divide the analysis in two cases. In particular, we consider $\Delta$ independent of $\epl$ and $\Delta=\epl^\zeta$ with $\zeta>0$. In the first case (Proposition \ref{pro:continuity_G_tilde_0_D}) data are sampled at the homogenized regime ignoring the fact that the they are generated by a multiscale model, while in the second case (Proposition \ref{pro:continuity_G_tilde_0_De}) the distance between two consecutive observations is proportional to the multiscale parameter and thus data are sampled at the multiscale regime preserving the multiscale structure of the full path.

\begin{proposition} \label{pro:continuity_G_tilde_0_D}
Let the functions $\widetilde{\mathcal G}_J \colon (0,\infty) \times \mathcal A \to \R^M$ and $\widetilde{\mathcal H}_J, \colon (0,\infty) \times \mathcal A \to \R^{M \times M}$ be defined in \eqref{eq:def_G_tilde} and \eqref{eq:def_H_tilde} and let $\Delta$ be independent of $\epl$. Under Assumption \ref{ass:beta_functions} and for any $a^* \in \mathcal A$ we have
\begin{equation}
\begin{aligned}
(i)& \lim_{(\epl,a) \to (0,a^*)} \widetilde{\mathcal G}_J(\epl,a) = \frac1\Delta \sum_{j=1}^J \E^{\widetilde \rho^0} \left[ g_j \left( X_0^0, X_\Delta^0, \widetilde Z_0^0; a^* \right) \right], \\
(ii)& \lim_{(\epl,a) \to (0,a^*)} \widetilde{\mathcal H}_J(\epl,a) = \frac1\Delta \sum_{j=1}^J \E^{\widetilde \rho^0} \left[ h_j \left( X_0^0, X_\Delta^0, \widetilde Z_0^0; a^* \right) \right].
\end{aligned}
\end{equation}
\end{proposition}
\begin{proof}
We only prove the statement for $\widetilde{\mathcal G}_J$, then the argument is similar for $\widetilde{\mathcal H}_J$. By the triangle inequality we have
\begin{equation}
\norm{\widetilde{\mathcal G}_J(\epl,a) - \frac1\Delta \sum_{j=1}^J \E^{\widetilde \rho^0} \left[ g_j \left( X_0^0, X_\Delta^0, \widetilde Z_0^0; a^* \right) \right]} \le Q_1(\epl,a) + Q_2(\epl),
\end{equation}
where
\begin{equation}
Q_1(\epl,a) = \norm{\widetilde{\mathcal G}_J(\epl,a) - \widetilde{\mathcal G}_J(\epl,a^*)},
\end{equation}
which vanishes due to the first step of the proof of Proposition \ref{pro:continuity_G_tilde} and
\begin{equation}
Q_2(\epl) = \norm{\frac1\Delta \sum_{j=1}^J \E^{\widetilde \rho^\epl} \left[ g_j \left( X_0^\epl, X_\Delta^\epl, \widetilde Z_0^\epl; a^* \right) \right] - \frac1\Delta \sum_{j=1}^J \E^{\widetilde \rho^0} \left[ g_j \left( X_0^0, X_\Delta^0, \widetilde Z_0^0; a^* \right) \right]}.
\end{equation}
Let us remark that the convergence in law of the joint process $\{(\widetilde X^\epl_n, \widetilde Z^\epl_n)\}_{n=0}^N$ to the joint process $\{(\widetilde X^0_n, \widetilde Z^0_n)\}_{n=0}^N$ by Lemma \ref{lem:ergodicity_XZ_tilde} implies the convergence in law of the triple $(X_0^\epl, X_\Delta^\epl, \widetilde Z_0^\epl)$ to the triple $(X_0^0, X^0_\Delta, \widetilde Z^0_0)$ since $\widetilde X_0^\epl = X_0^\epl$, $\widetilde X_1^\epl = X_\Delta^\epl$ and $\widetilde X_0^0 = X_0^0$, $\widetilde X_1^0 = X_\Delta^0$. Therefore we have
\begin{equation}
\lim_{\epl\to0} Q_2(\epl) \le \lim_{\epl\to0} \frac1\Delta \sum_{j=1}^J \norm{\E^{\widetilde \rho^\epl} \left[ g_j \left( X_0^\epl, X_\Delta^\epl, \widetilde Z_0^\epl; a^* \right) \right] - \E^{\widetilde \rho^0} \left[ g_j \left( X_0^0, X_\Delta^0, \widetilde Z_0^0; a^* \right) \right]} = 0,
\end{equation}
which implies the desired result.
\end{proof}

\begin{remark} \label{rem:continuity_G_hat_0_D}
Similar results to Proposition \ref{pro:continuity_G_tilde} and Proposition \ref{pro:continuity_G_tilde_0_D} can be shown for the estimator without filtered data. In particular we have that $\widehat{\mathcal G}_J(\epl,a)$ and $\widehat{\mathcal H}_J(\epl,a)$ are continuous in $(0,\infty) \times \mathcal A$ and
\begin{equation}
\begin{aligned}
(i)& \lim_{(\epl,a) \to (0,a^*)} \widehat{\mathcal G}_J(\epl,a) = \frac1\Delta \sum_{j=1}^J \E^{\varphi^0} \left[ g_j \left( X_0^0, X_\Delta^0, X_0^0; a^* \right) \right], \\
(ii)& \lim_{(\epl,a) \to (0,a^*)} \widehat{\mathcal H}_J(\epl,a) = \frac1\Delta \sum_{j=1}^J \E^{\varphi^0} \left[ h_j \left( X_0^0, X_\Delta^0, X_0^0; a^* \right) \right].
\end{aligned}
\end{equation}
Since the proof is analogous, we do not report here the details.
\end{remark}

\begin{proposition} \label{pro:continuity_G_tilde_0_De}
Let the functions $\widetilde{\mathcal G}_J \colon (0,\infty) \times \mathcal A \to \R^M$ and $\widetilde{\mathcal H}_J, \colon (0,\infty) \times \mathcal A \to \R^{M \times M}$ be defined in \eqref{eq:def_G_tilde} and \eqref{eq:def_H_tilde} and let $\Delta=\epl^\zeta$ with $\zeta>0$ and $\zeta \neq 1$, $\zeta \neq 2$. Under Assumption \ref{ass:beta_functions} and for any $a^* \in \mathcal A$ we have
\begin{enumerate}
\item $\lim_{(\epl,a) \to (0,a^*)} \widetilde{\mathcal G}_J(\epl,a) = \widetilde{\mathfrak g}_J^0(a^*)$, where
\begin{equation}
\widetilde{\mathfrak g}_J^0(a) \defeq \sum_{j=1}^J \E^{\rho^0} \left[ \beta_j(Z_0^0;a) \left( \diffL_A \phi_j(X_0^0;a) + \lambda_j(a) \phi_j(X_0^0;a) \right) \right],
\end{equation}
\item $\lim_{(\epl,a) \to (0,a^*)} \widetilde{\mathcal H}_J(\epl,a) = \widetilde{\mathfrak h}_J^0(a^*)$, where
\begin{equation}
\begin{aligned}
\widetilde{\mathfrak h}_J^0(a) &\defeq \sum_{j=1}^J \E^{\rho^0} \left[ \dot \beta_j(Z_0^0;a) \left( \diffL_A \phi_j(X_0^0;a) + \lambda_j(a) \phi_j(X_0^0;a) \right) \right] \\
&\quad + \sum_{j=1}^J \E^{\rho^0} \left[ \beta_j(Z_0^0;a) \otimes \left( \diffL_A \dot \phi_j(X_0^0;a) + \lambda_j(a) \dot \phi_j(X_0^0;a) \right) \right] \\
&\quad + \sum_{j=1}^J \E^{\rho^0} \left[ \beta_j(Z_0^0;a) \phi_j(X_0^0;a) \right] \otimes \dot \lambda_j(a),
\end{aligned}
\end{equation}
\end{enumerate}
where the generator $\diffL_A$ is defined in \eqref{eq:generator}.
\end{proposition}
\begin{proof}
We only prove the statement for $\widetilde{\mathcal G}_J$, then the argument is similar for $\widetilde{\mathcal H}_J$. By the triangle inequality we have
\begin{equation}
\norm{\widetilde{\mathcal G}_J(\epl,a) - \widetilde{\mathfrak g}_J^0(a^*)} \le \norm{\widetilde{\mathcal G}_J(\epl,a) - \widetilde{\mathcal G}_J(\epl,a^*)} + \norm{\widetilde{\mathcal G}_J(\epl,a^*) - \widetilde{\mathfrak g}_J^0(a^*)} \eqdef Q_1(\epl,a) + Q_2(\epl),
\end{equation}
then we need to show that the two terms vanish and we distinguish two cases. \\
\textbf{Case 1:} $\zeta \in (0,1)$. \\
Applying Lemma \ref{lem:expansion_fXD} to the functions $\phi_j(\cdot;a^*)$ for all $j=1,\dots,J$ and noting that
\begin{equation}
\E^{\widetilde \rho^\epl} \left[ \beta_j(\widetilde Z_0^\epl;a^*) \int_0^\Delta \phi_j'(X_t^\epl;a^*) (1+\Phi'(Y_t^\epl)) \dd W_t \right] = 0,
\end{equation}
since 
\begin{equation}
M_s \defeq \int_0^s \phi_j'(X_t^\epl;a^*) (1+\Phi'(Y_t^\epl)) \dd W_t
\end{equation}
is a martingale with $M_0=0$, we have
\begin{equation}
\begin{aligned}
\widetilde{\mathcal G}_J(\epl,a^*) &= \frac1\Delta \sum_{j=1}^J \E^{\widetilde \rho^\epl} \left[ \beta_j(\widetilde Z_0^\epl;a^*) \left( \phi_j(X_\Delta^\epl;a^*) - e^{-\lambda_j(a^*)\Delta} \phi_j(X_0^\epl;a^*) \right) \right] \\
&= \frac{1 - e^{-\lambda_j(a^*)\Delta}}{\Delta} \sum_{j=1}^J \E^{\widetilde \rho^\epl} \left[ \beta_j(\widetilde Z_0^\epl;a^*) \phi_j(X_0^\epl;a^*) \right] + \sum_{j=1}^J \frac1\Delta \E^{\widetilde \rho^\epl} \left[ \beta_j(\widetilde Z_0^\epl;a^*) R(\epl,\Delta) \right] \\
&\quad + \sum_{j=1}^J \E^{\widetilde \rho^\epl} \left[ \beta_j(\widetilde Z_0^\epl;a^*) \left( \Sigma \phi_j''(X_0^\epl;a^*) - A \cdot V'(X_0^\epl) \phi_j'(X_0^\epl;a^*) \right) \right] \\
&\eqdef I_1^\epl + I_2^\epl + I_3^\epl,
\end{aligned}
\end{equation}
where $R(\epl,\Delta)$ satisfies for a constant $C>0$ independent of $\epl$ and $\Delta$ and for all $p\ge1$
\begin{equation} \label{eq:bound_R_lemma}
\left( \E^{\widetilde \rho^\epl} \abs{R(\epl,\Delta)}^p \right)^{1/p} \le C(\epl + \Delta^{3/2}).
\end{equation}
We now study the three terms separately. First, by Cauchy-Schwarz inequality, since $\beta_j(\cdot;a^*)$ is polynomially bounded, $\widetilde Z_0^\epl$ has bounded moments of any order by Lemma \ref{lem:Ztilde_bounded_moments} and due to \eqref{eq:bound_R_lemma} we obtain
\begin{equation} \label{eq:limit_I2}
\norm{I_2^\epl} \le C \left( \epl \Delta^{-1} + \Delta^{1/2} \right).
\end{equation}
Let us now focus on $I_1^\epl$ for which we have
\begin{equation}
I_1^\epl = \frac{1 - e^{-\lambda_j(a^*)\Delta}}{\Delta} \sum_{j=1}^J \left( \E^{\rho^\epl} \left[ \beta_j(Z_0^\epl;a^*) \phi_j(X_0^\epl;a^*) \right] + \E \left[ \left( \beta_j(\widetilde Z_0^\epl;a^*) - \beta_j(Z_0^\epl;a^*) \right) \phi_j(X_0^\epl;a^*) \right] \right),
\end{equation}
where $Z_0^\epl$ is distributed according to the invariant measure $\rho^\epl$ of the continuous process $(X_t^\epl,Z_t^\epl)$ and
\begin{equation} \label{eq:limit_eigenvalue}
\lim_{\epl\to0} \frac{1 - e^{-\lambda_j(a^*)\Delta}}{\Delta} = \lambda_j(a^*).
\end{equation}
By the mean value theorem for vector-valued functions we have
\begin{equation}
\E \left[ ( \beta_j(\widetilde Z_0^\epl;a^*) - \beta_j(Z_0^\epl;a^*) ) \phi_j(X_0^\epl;a^*) \right] = \E \left[\int_0^1 \beta_j'(Z_0^\epl + t(\widetilde Z_0^\epl - Z_0^\epl);a^*) \dd t \; (\widetilde Z_0^\epl - Z_0^\epl) \phi_j(X_0^\epl;a^*) \right],
\end{equation}
and since $\beta_j'(\cdot;a^*),\phi_j(\cdot;a^*)$ are polynomially bounded, $X_0^\epl$, $Z_0^\epl$, $\widetilde Z_0^\epl$ have bounded moments of any order respectively by \cite[Corollary 5.4]{PaS07}, \cite[Lemma C.1]{AGP20} and Lemma \ref{lem:Ztilde_bounded_moments} and applying Hölder's inequality and Corollary \ref{cor:asymptotic_distances_Z} we obtain
\begin{equation} \label{eq:bound_remainder_eigenvalue}
\norm{\E \left[ \left( \beta_j(\widetilde Z_0^\epl;a^*) - \beta_j(Z_0^\epl;a^*) \right) \phi_j(X_0^\epl;a^*) \right]} \le C \left( \Delta^{1/2} + \epl \right).
\end{equation}
Moreover, notice that by homogenization theory (see \cite[Section 3.2]{AGP20}) the joint process $(X_0^\epl, Z_0^\epl)$ converges in law to the joint process $(X_0^0, Z_0^0)$ and therefore
\begin{equation}
\lim_{\epl\to0} \E^{\rho^\epl} \left[ \beta_j(Z_0^\epl;a^*) \phi_j(X_0^\epl;a^*) \right] = \E^{\rho^0} \left[ \beta_j(Z_0^0;a^*) \phi_j(X_0^0;a^*) \right],
\end{equation}
which together with \eqref{eq:limit_eigenvalue} and \eqref{eq:bound_remainder_eigenvalue} yields
\begin{equation} \label{eq:limit_I1}
\lim_{\epl\to0} I_1^\epl = \sum_{j=1}^J \lambda_j(a^*) \E^{\rho^0} \left[ \beta_j(Z_0^0;a^*) \phi_j(X_0^0;a^*) \right].
\end{equation}
We now consider $I_3^\epl$ and we follow an argument similar to $I_2^\epl$. We first have
\begin{equation}
\begin{aligned}
I_3^\epl &= \sum_{j=1}^J \E^{\rho^\epl} \left[ \beta_j(Z_0^\epl;a^*) \left( \Sigma \phi_j''(X_0^\epl;a^*) - A \cdot V'(X_0^\epl) \phi_j'(X_0^\epl;a^*) \right) \right] \\
&\quad + \sum_{j=1}^J \E \left[ \left( \beta_j(\widetilde Z_0^\epl;a^*) - \beta_j(Z_0^\epl;a^*) \right) \left( \Sigma \phi_j''(X_0^\epl;a^*) - A \cdot V'(X_0^\epl) \phi_j'(X_0^\epl;a^*) \right) \right] \\
&\eqdef I_{3,1}^\epl + I_{3,2}^\epl,
\end{aligned}
\end{equation}
where the first term in the right-hand side converges due to homogenization theory and the second one is bounded by
\begin{equation}
\norm{I_{3,2}^\epl} \le C \left( \Delta^{1/2} + \epl \right).
\end{equation}
Therefore, we obtain
\begin{equation}
\lim_{\epl\to0} I_3^\epl = \sum_{j=1}^J \E^{\rho^0} \left[ \beta_j(Z_0^0;a^*) \left( \Sigma \phi_j''(X_0^0;a^*) - A \cdot V'(X_0^0) \phi_j'(X_0^0;a^*) \right) \right],
\end{equation}
which together with \eqref{eq:limit_I2} and \eqref{eq:limit_I1} implies
\begin{equation} \label{eq:goal_e_case1}
\lim_{\epl\to0} \widetilde{\mathcal G}_J(\epl,a^*) = \sum_{j=1}^J \E^{\rho^0} \left[ \beta_j(Z_0^0;a) \left( \Sigma \phi_j''(X_0^0;a^*) - A \cdot V'(X_0^0) \phi_j'(X_0^0;a^*) + \lambda_j(a^*) \phi_j(X_0^0;a^*) \right) \right],
\end{equation}
which shows that $Q_2(\epl)$ vanishes as $\epl$ goes to zero. Let us now consider $Q_1(\epl,a)$. Following the first step of the proof of Proposition \ref{pro:continuity_G_tilde} we have
\begin{equation}
\begin{aligned}
Q_1(\epl,a) &\le \frac1\Delta \sum_{j=1}^J \norm{\E^{\widetilde \rho^\epl} \left[ g_j(X_0^\epl, X_\Delta^\epl, \widetilde Z_0^\epl; a) \right] - \E^{\widetilde \rho^\epl} \left[ g_j(X_0^\epl, X_\Delta^\epl, \widetilde Z_0^\epl; a^*) \right]} \\
&\le \sum_{j=1}^J \norm{\frac1\Delta \E^{\widetilde \rho^\epl} \left[ h_j(X_0^\epl, X_\Delta^\epl, \widetilde Z_0^\epl; \widetilde a) \right]} \norm{(a-a^*)},
\end{aligned}
\end{equation}
where $\widetilde a$ assumes values in the line connecting $a$ and $a^*$, and repeating the same computation as above we obtain
\begin{equation}
Q_1(\epl,a) \le C \norm{a-a^*},
\end{equation}
which together with \eqref{eq:goal_e_case1} gives the desired result. \\
\textbf{Case 2:} $\zeta \in (1,2) \cup (2,\infty)$. \\
Let $Z_0^\epl$ be distributed according to the invariant measure $\rho^\epl$ of the continuous process $(X_t^\epl,Z_t^\epl)$ and define
\begin{equation}
\begin{aligned}
\widetilde R(\epl,\Delta) &\defeq \frac1\Delta \sum_{j=1}^J \E^{\widetilde \rho^\epl} \left[ g_j(X_0^\epl, X_\Delta^\epl, \widetilde Z_0^\epl; a^*) \right] - \frac1\Delta \sum_{j=1}^J \E^{\rho^\epl} \left[ g_j(X_0^\epl, X_\Delta^\epl, Z_0^\epl; a^*) \right] \\
&= \frac1\Delta \sum_{j=1}^J \E \left[ \left( \beta_j(\widetilde Z_0^\epl;a^*) - \beta_j(Z_0^\epl;a^*) \right) \left( \phi_j(X_\Delta^\epl;a^*) - e^{-\lambda_j(a^*)\Delta} \phi_j(X_0^\epl;a^*) \right) \right].
\end{aligned}
\end{equation}
Then we have
\begin{equation} \label{eq:def_Qje}
\widetilde{\mathcal G}_J(\epl,a^*) = \sum_{j=1}^J \frac1\Delta \E^{\rho^\epl} \left[ g_j(X_0^\epl, X_\Delta^\epl, Z_0^\epl; a^*) \right] + \widetilde R(\epl,\Delta) \eqdef \sum_{j=1}^J Q_j^\epl + \widetilde R(\epl,\Delta),
\end{equation}
and we first bound the remainder $\widetilde R(\epl,\Delta)$. Applying Itô's lemma to the process $X_t^\epl$ with the functions $\phi_j(\cdot;a^*)$ for each $j=1,\dots,J$ we have
\begin{equation} \label{eq:ito_phi}
\begin{aligned}
\phi_j(X_\Delta^\epl;a^*) &= \phi_j(X_0^\epl;a^*) - \int_0^\Delta \alpha \cdot V'(X_t^\epl) \phi_j'(X_t^\epl;a^*) \dd t - \int_0^\Delta \frac1\epl p' \left( \frac{X_t^\epl}{\epl} \right) \phi_j'(X_t^\epl;a^*) \dd t \\
&\quad + \sigma \int_0^\Delta \phi_j''(X_t^\epl;a^*) \dd t + \sqrt{2\sigma} \int_0^\Delta \phi_j'(X_t^\epl;a^*) \dd W_t,
\end{aligned}
\end{equation}
and observing that 
\begin{equation} \label{eq:martingale0}
\E \left[ \left( \beta_j(\widetilde Z_0^\epl;a^*) - \beta_j(Z_0^\epl;a^*) \right) \int_0^\Delta \phi_j'(X_t^\epl;a^*) \dd W_t \right] = 0,
\end{equation}
since 
\begin{equation}
M_s = \int_0^s \phi_j'(X_t^\epl;a^*) \dd W_t
\end{equation}
is a martingale with $M_0=0$, we obtain
\begin{equation}
\begin{aligned}
\widetilde R(\epl,\Delta) &= \sum_{j=1}^J \frac{1-e^{-\lambda_j(a^*)\Delta}}{\Delta} \E \left[ \left( \beta_j(\widetilde Z_0^\epl;a^*) - \beta_j(Z_0^\epl;a^*) \right) \phi_j(X_0^\epl;a^*) \right] \\
&\quad + \sum_{j=1}^J \frac1\Delta \int_0^\Delta \E \left[ \left( \beta_j(\widetilde Z_0^\epl;a^*) - \beta_j(Z_0^\epl;a^*) \right) \left( \sigma \phi_j''(X_t^\epl;a^*) - \alpha \cdot V'(X_t^\epl) \phi_j'(X_t^\epl;a^*) \right) \right] \dd t \\
&\quad - \sum_{j=1}^J \frac{1}{\epl\Delta} \int_0^\Delta \E \left[ \left( \beta_j(\widetilde Z_0^\epl;a^*) - \beta_j(Z_0^\epl;a^*) \right) p' \left( \frac{X_t^\epl}{\epl} \right) \phi_j'(X_t^\epl;a^*) \right] \dd t \\
&\eqdef \widetilde R_1(\epl,\Delta) + \widetilde R_2(\epl,\Delta) + \widetilde R_3(\epl,\Delta).
\end{aligned}
\end{equation}
By the mean value theorem for vector-valued functions we have
\begin{equation}
\E \left[ ( \beta_j(\widetilde Z_0^\epl;a^*) - \beta_j(Z_0^\epl;a^*) ) \phi_j(X_0^\epl;a^*) \right] = \E \left[ \int_0^1 \beta_j'(Z_0^\epl + t(\widetilde Z_0^\epl - Z_0^\epl);a^*) \dd t \; ( \widetilde Z_0^\epl - Z_0^\epl ) \phi_j(X_0^\epl;a^*) \right],
\end{equation}
and since $\beta_j'(\cdot;a^*),\phi_j(\cdot;a^*)$ are polynomially bounded, $X_0^\epl$, $Z_0^\epl$, $\widetilde Z_0^\epl$ have bounded moments of any order respectively by \cite[Corollary 5.4]{PaS07}, \cite[Lemma C.1]{AGP20} and Lemma \ref{lem:Ztilde_bounded_moments} and applying Hölder's inequality we obtain
\begin{equation} \label{eq:bound_intermediate_Rtilde1}
\norm{\widetilde R_1(\epl,\Delta)} \le C \left( \E \abs{\widetilde Z_0^\epl - Z_0^\epl}^{2} \right)^{1/2},
\end{equation}
for a constant $C>0$ independent of $\epl$ and $\Delta$. We repeat a similar argument for $\widetilde R_2(\epl,\Delta)$ and $\widetilde R_3(\epl,\Delta)$ to get
\begin{equation}
\norm{\widetilde R_2(\epl,\Delta)} \le C \left( \E \abs{\widetilde Z_0^\epl - Z_0^\epl}^{2} \right)^{1/2} \quad \text{and} \quad \norm{\widetilde R_3(\epl,\Delta)} \le C \epl^{-1} \left( \E \abs{\widetilde Z_0^\epl - Z_0^\epl}^{2} \right)^{1/2},
\end{equation}
which together with \eqref{eq:bound_intermediate_Rtilde1} yield
\begin{equation} \label{eq:bound_Rtilde}
\norm{\widetilde R(\epl,\Delta)} \le C \left( \E \abs{\widetilde Z_0^\epl - Z_0^\epl}^{2} \right)^{1/2} \left( 1 + \epl^{-1} \right).
\end{equation}
Moreover, applying Lemma \ref{lem:expansion_fXD} and proceeding similarly to the first part of the first case of the proof we have
\begin{equation}
\norm{\widetilde R(\epl,\Delta)} \le C \left( \E \abs{\widetilde Z_0^\epl - Z_0^\epl}^{2} \right)^{1/2} \left( 1 + \epl\Delta^{-1} + \Delta^{1/2} \right),
\end{equation}
which together with \eqref{eq:bound_Rtilde} and Corollary \ref{cor:asymptotic_distances_Z} implies
\begin{equation} \label{eq:final_bound_Rtilde}
\begin{aligned}
\norm{\widetilde R(\epl,\Delta)} &\le C \left( \E \abs{\widetilde Z_0^\epl - Z_0^\epl}^{2} \right)^{1/2} \left( 1 + \min \{ \epl^{-1}, \epl\Delta^{-1} + \Delta^{1/2} \} \right) \\
&\le C \left( \Delta^{1/2} + \min \{ \epl, \epl^{-1}\Delta \} \right) \left( 1 + \min \{ \epl^{-1}, \epl\Delta^{-1} + \Delta^{1/2} \} \right).
\end{aligned}
\end{equation}
Let us now consider $Q_j^\epl$. Replacing equation \eqref{eq:ito_phi} into the definition of $Q_j^\epl$ in \eqref{eq:def_Qje} and observing that similarly to \eqref{eq:martingale0} it holds
\begin{equation}
\E^{\rho^\epl} \left[ \beta_j(Z_0^\epl;a^*) \int_0^\Delta \phi_j'(X_t^\epl;a^*) \dd W_t \right] = 0,
\end{equation}
we obtain
\begin{equation}
\begin{aligned}
Q_j^\epl &= \frac{1-e^{-\lambda_j(a^*)}}{\Delta} \E^{\rho^\epl} \left[ \beta_j(Z_0^\epl;a^*) \phi_j(X_0^\epl;a^*) \right] \\
&\quad - \frac1\Delta \left( \int_0^\Delta \E^{\rho^\epl} \left[ \left( \beta_j(Z_0^\epl;a^*) \otimes V'(X_t^\epl) \right) \phi_j'(X_t^\epl;a^*) \right] \dd t \right) \alpha \\
&\quad - \frac1\Delta \int_0^\Delta \E^{\rho^\epl} \left[ \beta_j(Z_0^\epl;a^*) \frac1\epl p' \left( \frac{X_t^\epl}{\epl} \right) \phi_j'(X_t^\epl;a^*) \right] \dd t + \frac\sigma\Delta \int_0^\Delta \E^{\rho^\epl} \left[ \beta_j(Z_0^\epl;a^*) \phi_j''(X_t^\epl;a^*) \right] \dd t.
\end{aligned}
\end{equation}
We rewrite $\beta_j(Z_0^\epl;a^*)$ inside the integrals employing equation \eqref{eq:systemSDE_MS} and Itô's lemma
\begin{equation}
\beta_j(Z_0^\epl;a^*) = \beta_j(Z_t^\epl;a^*) - \int_0^t \beta_j'(Z_s^\epl;a^*) \left( X_s^\epl - Z_s^\epl \right) \dd s,
\end{equation}
hence due to stationarity we have
\begin{equation} \label{eq:decomposition_Q12}
Q_j^\epl = Q_{j,1}^\epl + Q_{j,2}^{\epl},
\end{equation}
where
\begin{equation}
\begin{aligned}
Q_{j,1}^\epl &= \frac{1-e^{-\lambda_j(a^*)}}{\Delta} \E^{\rho^\epl} \left[ \beta_j(Z_0^\epl;a^*) \phi_j(X_0^\epl;a^*) \right] - \E^{\rho^\epl} \left[ \left( \beta_j(Z_0^\epl;a^*) \otimes V'(X_0^\epl) \right) \phi_j'(X_0^\epl;a^*) \right] \alpha \\
&\quad - \E^{\rho^\epl} \left[ \beta_j(Z_0^\epl;a^*) \frac1\epl p' \left( \frac{X_0^\epl}{\epl} \right) \phi_j'(X_0^\epl;a^*) \right] + \sigma \E^{\rho^\epl} \left[ \beta_j(Z_0^\epl;a^*) \phi_j''(X_0^\epl;a^*) \right]
\end{aligned}
\end{equation}
and
\begin{equation}
\begin{aligned}
Q_{j,2}^\epl &= \frac1\Delta \left( \int_0^\Delta \int_0^t \E^{\rho^\epl} \left[ (\beta_j'(Z_s^\epl;a^*) \otimes V'(X_t^\epl)) \phi_j'(X_t^\epl;a^*) (X_s^\epl - Z_s^\epl) \right] \dd s \dd t \right) \alpha \\
&\quad + \frac1\Delta \int_0^\Delta \int_0^t \E^{\rho^\epl} \left[ \beta_j'(Z_s^\epl;a^*) \frac1\epl p' \left( \frac{X_t^\epl}{\epl} \right) \phi_j'(X_t^\epl;a^*) (X_s^\epl - Z_s^\epl) \right] \dd s \dd t \\
&\quad - \frac\sigma\Delta \int_0^\Delta \int_0^t \E^{\rho^\epl} \left[ \beta_j'(Z_s^\epl;a^*) \phi_j''(X_t^\epl;a^*) (X_s^\epl - Z_s^\epl) \right] \dd s \dd t.
\end{aligned}
\end{equation}
Since $\phi_j'(\cdot;a^*), \phi_j''(\cdot;a^*)$ and $\beta_j'(\cdot;a^*)$ are polynomially bounded, $p'$ is bounded and $X_t^\epl$ and $Z_t^\epl$ have bounded moments of any order respectively by \cite[Corollary 5.4]{PaS07} and \cite[Lemma C.1]{AGP20}, $Q_{j,2}^\epl$ is bounded by
\begin{equation} \label{eq:bound_Qj2}
\norm{Q_{j,2}^\epl} \le C \left( \Delta + \epl^{-1}\Delta \right).
\end{equation}
Let us now move to $Q_{j,1}^\epl$ and let us define the functions
\begin{equation}
\eta^\epl(x,z) \defeq \frac{\rho^\epl(x,z)}{\varphi^\epl(x)} \qquad \text{and} \qquad \eta^0(x,z) \defeq \frac{\rho^0(x,z)}{\varphi^0(x)},
\end{equation}
where $\rho^\epl$ and $\rho^0$ are respectively the densities with respect to the Lebesgue measure of the invariant distributions of the joint processes $(X_t^\epl,Z_t^\epl)$ and $(X_t^0,Z_t^0)$ and $\varphi^\epl$ and $\varphi^0$ are their marginals with respect to the first component. Integrating by parts we have
\begin{equation}
\begin{aligned}
\E^{\rho^\epl} \left[ \beta_j(Z_0^\epl;a^*) \frac1\epl p' \left( \frac{X_0^\epl}{\epl} \right) \phi_j'(X_0^\epl;a^*) \right] &= \int_\R \int_\R \beta_j(z;a^*) \frac1\epl p' \left( \frac x\epl \right) \phi_j'(x;a^*) \rho^\epl(x,z) \dd x \dd z \\
&\hspace{-2cm} = -\sigma \int_\R \int_\R \frac{1}{C_{\varphi^\epl}} \beta_j(z;a^*) \frac{\d}{\dd x} \left( e^{-\frac1\sigma p \left( \frac x\epl \right)} \right) \phi_j'(x;a^*) e^{-\frac1\sigma \alpha \cdot V(x)} \eta^\epl(x,z) \dd x \dd z \\
&\hspace{-2cm} = \sigma \int_\R \int_\R \frac{1}{C_{\varphi^\epl}} \beta_j(z;a^*) \frac{\partial}{\partial x} \left( \phi_j'(x;a^*) e^{-\frac1\sigma \alpha \cdot V(x)} \eta^\epl(x,z) \right) e^{-\frac1\sigma p \left( \frac x\epl \right)} \dd x \dd z,
\end{aligned}
\end{equation}
which implies
\begin{equation}
\begin{aligned}
\E^{\rho^\epl} \left[ \beta_j(Z_0^\epl;a^*) \frac1\epl p' \left( \frac{X_0^\epl}{\epl} \right) \phi_j'(X_0^\epl;a^*) \right] &= \sigma \E^{\rho^\epl} \left[ \beta_j(Z_0^\epl;a^*) \phi_j''(X_0^\epl;a^*) \right] \\
&\quad - \E^{\rho^\epl} \left[ (\beta_j(Z_0^\epl;a^*) \otimes V(X_0^\epl)) \phi_j'(X_0^\epl;a^*) \right] \alpha \\
&\quad + \sigma \int_\R \int_\R \beta_j(z;a^*) \phi_j'(x;a^*) \varphi^\epl(x) \frac{\partial}{\partial x} \eta^\epl(x,z) \dd x \dd z.
\end{aligned}
\end{equation}
Employing the last equation in the proof of Lemma 3.5 in \cite{AGP20} with $\delta=1$ and $f(x,z) = \beta_j(z;a^*) \phi_j'(x;a^*)$ we have
\begin{equation} \label{eq:magic}
\sigma \int_\R \int_\R \beta_j(z;a^*) \phi_j'(x;a^*) \varphi^\epl(x) \frac{\partial}{\partial x} \eta^\epl(x,z) \dd x \dd z = \E^{\rho^\epl} \left[ \beta_j'(Z_0^\epl;a^*) \phi_j(X_0^\epl;a^*) (X_0^\epl - Z_0^\epl) \right],
\end{equation}
and thus we obtain
\begin{equation}
Q_{j,1}^\epl = \frac{1-e^{-\lambda_j(a^*)}}{\Delta} \E^{\rho^\epl} \left[ \beta_j(Z_0^\epl;a^*) \phi_j(X_0^\epl;a^*) \right] - \E^{\rho^\epl} \left[ \beta_j'(Z_0^\epl;a^*) \phi_j(X_0^\epl;a^*) (X_0^\epl - Z_0^\epl) \right].
\end{equation}
Letting $\epl$ go to zero and due to homogenization theory, it follows
\begin{equation}
\lim_{\epl\to0} Q_{j,1}^\epl = \lambda_j(a^*) \E^{\rho^0} \left[ \beta_j(Z_0^0;a^*) \phi_j(X_0^0;a^*) \right] - \E^{\rho^0} \left[ \beta_j'(Z_0^0;a^*) \phi_j(X_0^0;a^*) (X_0^0 - Z_0^0) \right],
\end{equation}
then applying formula \eqref{eq:magic} for the homogenized equation, i.e. with $p(y)=0$ and $\alpha$ and $\sigma$ replaced by $A$ and $\Sigma$, and integrating by parts we have
\begin{equation}
\begin{aligned}
\E^{\rho^0} \left[ \beta_j'(Z_0^0;a^*) \phi_j(X_0^0;a^*) (X_0^0 - Z_0^0) \right] &= \Sigma \int_\R \int_\R \beta_j(z;a^*) \phi_j'(x;a^*) \varphi^0(x) \frac{\partial}{\partial x} \eta^0(x,z) \dd x \dd z \\
&= - \Sigma \int_\R \int_\R \beta_j(z;a^*) \frac{\d}{\dd x} \left( \phi_j'(x;a^*) \varphi^0(x) \right) \eta^0(x,z) \dd x \dd z \\
&= \E^{\rho^0} \left[ \beta_j(Z_0^0;a^*) \left( \Sigma \phi_j''(X_0^0;a^*) - A \cdot V'(X_0^0) \phi_j'(X_0^0;a^*) \right) \right].
\end{aligned}
\end{equation}
Therefore, we obtain
\begin{equation}
\lim_{\epl\to0} Q_{j,1}^\epl = \E^{\rho^0} \left[ \beta_j(Z_0^0;a^*) \left( \Sigma \phi_j''(X_0^0;a^*) - A \cdot V'(X_0^0) \phi_j'(X_0^0;a^*) + \lambda_j(a^*) \phi_j(X_0^0;a^*) \right) \right],
\end{equation}
which together with \eqref{eq:def_Qje}, \eqref{eq:decomposition_Q12} and bounds \eqref{eq:final_bound_Rtilde} and \eqref{eq:bound_Qj2} implies that $Q_2(\epl)$ vanishes as $\epl$ goes to zero. Finally, analogously to the first case we can show that also $Q_1(\epl,a)$ vanishes, concluding the proof.
\end{proof}

\begin{remark} \label{rem:continuity_G_hat_0_De}
A similar result to Proposition \ref{pro:continuity_G_tilde_0_De} can be shown for the estimator without filtered data only if $\zeta \in (0,1)$, i.e. the first case in the proof. In particular, we have
\begin{enumerate}[leftmargin=*]
\item $\lim_{(\epl,a) \to (0,a^*)} \widehat{\mathcal G}_J(\epl,a) = \widehat{\mathfrak g}_J^0(a^*)$, where
\begin{equation}
\widehat{\mathfrak g}_J^0(a) \defeq \sum_{j=1}^J \E^{\varphi^0} \left[ \beta_j(X_0^0;a) \left( \diffL_A \phi_j(X_0^0;a) + \lambda_j(a) \phi_j(X_0^0;a) \right) \right],
\end{equation}
\item $\lim_{(\epl,a) \to (0,a^*)} \widehat{\mathcal H}_J(\epl,a) = \widehat{\mathfrak h}_J^0(a^*)$, where
\begin{equation}
\begin{aligned}
\widehat{\mathfrak h}_J^0(a) &\defeq \sum_{j=1}^J \E^{\varphi^0} \left[ \dot \beta_j(X_0^0;a) \left( \diffL_A \phi_j(X_0^0;a) + \lambda_j(a) \phi_j(X_0^0;a) \right) \right] \\
&\quad + \sum_{j=1}^J \E^{\varphi^0} \left[ \beta_j(X_0^0;a) \otimes \left( \diffL_A \dot \phi_j(X_0^0;a) + \lambda_j(a) \dot \phi_j(X_0^0;a) \right) \right] \\
&\quad + \E^{\varphi^0} \left[ \beta_j(X_0^0;a) \phi_j(X_0^0;a) \right] \otimes \dot \lambda_j(a),
\end{aligned}
\end{equation}
\end{enumerate}
where the generator $\diffL_A$ is defined in \eqref{eq:generator}. Since the proof is analogous, we do not report here the details. On the other hand, if $\zeta > 2$ we can show that
\begin{enumerate}[leftmargin=*]
\item $\lim_{(\epl,a) \to (0,a^*)} \widehat{\mathcal G}_J(\epl,a) = \mathfrak g_J^0(a^*)$, where
\begin{equation} \label{eq:continuity_G_hat_0_De_different}
\mathfrak g_J^0(a) \defeq \sum_{j=1}^J \E^{\varphi^0} \left[ \beta_j(X_0^0;a) \left( \sigma \phi_j''(X_0^0;a) - \alpha \cdot V'(X_0^0) \phi_j'(X_0^0;a) + \lambda_j(a) \phi_j(X_0^0;a) \right) \right],
\end{equation}
\item $\lim_{(\epl,a) \to (0,a^*)} \widehat{\mathcal H}_J(\epl,a) = \mathfrak h_J^0(a^*)$, where
\begin{equation}
\begin{aligned}
\mathfrak h_J^0(a) &\defeq \sum_{j=1}^J \E^{\varphi^0} \left[ \dot \beta_j(X_0^0;a) \left( \sigma \phi_j''(X_0^0;a) - \alpha \cdot V'(X_0^0) \phi_j'(X_0^0;a) + \lambda_j(a) \phi_j(X_0^0;a) \right) \right] \\
&\quad + \sum_{j=1}^J \E^{\varphi^0} \left[ \beta_j(X_0^0;a) \otimes \left( \sigma \dot \phi_j''(X_0^0;a) - \alpha \cdot V'(X_0^0) \dot \phi_j'(X_0^0;a) + \lambda_j(a) \dot \phi_j(X_0^0;a) \right) \right] \\
&\quad + \sum_{j=1}^J \E^{\varphi^0} \left[ \beta_j(X_0^0;a) \phi_j(X_0^0;a) \right] \otimes \dot \lambda_j(a).
\end{aligned}
\end{equation}
\end{enumerate}
The proof is omitted since it is similar to the second case of the proof of Proposition \ref{pro:continuity_G_tilde_0_De}.
\end{remark}

\subsection{Proof of the main results} \label{sec:proof_main_results}

Let us remark that we aim to prove the asymptotic unbiasedness of the proposed estimators, i.e., their convergence to the homogenized drift coefficient $A$ as the the number of observations $N$ tends to infinity and the multiscale parameter $\epl$ vanishes. Therefore, we study the limit of the score functions and their Jacobian matrices as $N\to\infty$ and $\epl\to0$ evaluated in the desired limit point $A$. 

We first analyse the case $\Delta$ independent of $\epl$ and we consider the limit of Proposition \ref{pro:continuity_G_tilde_0_D} and Remark \ref{rem:continuity_G_hat_0_D} evaluated in $a^* = A$. Then due to equation \eqref{eq:martingale_formula} we get
\begin{equation} \label{eq:limit_score_A_ind}
\begin{aligned}
\frac1\Delta \sum_{j=1}^J \E^{\widetilde \rho^0} \left[ g_j \left( X_0^0, X_\Delta^0, \widetilde Z_0^0; A \right) \right] &= \frac1\Delta \sum_{j=1}^J \E^{\widetilde \rho^0} \left[ \beta_j(\widetilde Z_0^0;A) \left( \phi_j(X_\Delta^0;A) - e^{-\lambda_j(A)\Delta} \phi_j(X_0^0;A) \right) \right] \\
&\hspace{-2cm}= \frac1\Delta \sum_{j=1}^J \E^{\widetilde \rho^0} \left[ \beta_j(\widetilde Z_0^0;A) \left( \E \left[ \left. \phi_j(X_\Delta^0;A) \right| (X_0^0, \widetilde Z_0^0) \right] - e^{-\lambda_j(A)\Delta} \phi_j(X_0^0;A) \right) \right] \\
&\hspace{-2cm}= 0,
\end{aligned}
\end{equation}
and similarly we obtain
\begin{equation}
\frac1\Delta \sum_{j=1}^J \E^{\varphi^0} \left[ g_j \left( X_0^0, X_\Delta^0, X_0^0; A \right) \right] = 0.
\end{equation}
On the other hand, if $\Delta$ is a power of $\epl$ we study the limit of Proposition \ref{pro:continuity_G_tilde_0_De} and Remark \ref{rem:continuity_G_hat_0_De} evaluated in $a^* = A$ and by \eqref{eq:eigen_problem_general} we have
\begin{equation} \label{eq:limit_score_A_e}
\widetilde{\mathfrak g}_J^0(A) = 0 \qquad \text{and} \qquad \widehat{\mathfrak g}_J^0(A) = 0.
\end{equation}
Moreover, differentiating equation \eqref{eq:martingale_formula} with respect to $a$, we get
\begin{equation} \label{eq:diff_martingale:formula}
\begin{aligned}
\E \left[ \dot \phi_j(X_{t_n}(a);a) | X_{t_{n-1}}(a) = x \right] &= e^{-\lambda_j(a)\Delta} \dot \phi_j(x;a) - \dot \lambda_j(a)\Delta e^{-\lambda_j(a)\Delta} \phi_j(x;a) \\
&\quad - \E \left[\phi'_j(X_{t_n}(a);a) \nabla_a X_{t_n}(a) | X_{t_{n-1}}(a) = x \right],
\end{aligned}
\end{equation}
where the process $\nabla_a X_t(a)$ satisfies
\begin{equation}
\d \left( \nabla_a X_t(a) \right) = - V'(X_t) \dd t - a \cdot V''(X_t) \nabla_a X_t(a) \dd t.
\end{equation}
Therefore, due to \eqref{eq:martingale_formula} and \eqref{eq:diff_martingale:formula} we have
\begin{equation} \label{eq:limit_diff_score_tilde_A_ind}
\frac1\Delta \sum_{j=1}^J \E^{\widetilde \rho^0} \left[ h_j \left( X_0^0, X_\Delta^0, \widetilde Z_0^0; A \right) \right] = - \sum_{j=1}^J \E^{\widetilde \rho^0} \left[ \left( \beta_j(\widetilde Z_0^0;A) \otimes \nabla_a X_\Delta(A) \right) \phi_j'(X_\Delta^0;A) \right],
\end{equation}
and
\begin{equation} \label{eq:limit_diff_score_hat_A_ind}
\frac1\Delta \sum_{j=1}^J \E^{\varphi^0} \left[ h_j \left( X_0^0, X_\Delta^0, X_0^0; A \right) \right] = - \sum_{j=1}^J \E^{\varphi^0} \left[ \left( \beta_j(X_0^0;A) \otimes \nabla_a X_\Delta(A) \right) \phi_j'(X_\Delta^0;A) \right].
\end{equation}
Then due to Lemma \ref{lem:regularity_eigen_a} we can differentiate the eigenvalue problem \eqref{eq:eigen_problem_equation} with respect to $a$ and deduce that
\begin{equation}
\Sigma \dot \phi_j''(x;a) - a \cdot V'(x) \dot \phi_j'(x;a) + \lambda_j(a) \dot \phi_j(x;a) = V'(x) \phi_j'(x;a) - \dot \lambda_j \phi_j(x;a),
\end{equation}
where the dot denotes the gradient with respect to $a$ and which together with \eqref{eq:eigen_problem_equation} implies
\begin{equation} \label{eq:limit_diff_score_tilde_A_e}
\widetilde{\mathfrak h}_J^0(A) = \sum_{j=1}^J \E^{\rho^0} \left[ (\beta_j(Z_0^0;A) \otimes V'(X_0^0)) \phi_j'(X_0^0;A) \right],
\end{equation}
and
\begin{equation} \label{eq:limit_diff_score_hat_A_e}
\widehat{\mathfrak h}_J^0(A) = \sum_{j=1}^J \E^{\varphi^0} \left[ (\beta_j(X_0^0;A) \otimes V'(X_0^0)) \phi_j'(X_0^0;A) \right].
\end{equation}

Before showing the main results, we need two auxiliary lemmas, which in turn rely on the technical Assumption \ref{ass:for_Dini}, which can now be rewritten as
\begin{enumerate}
\item $\det \left( \frac1\Delta \sum_{j=1}^J \E^{\widetilde \rho^0} \left[ h_j \left( X_0^0, X_\Delta^0, \widetilde Z_0^0; A \right) \right] \right) \neq 0$,
\item $\det \left( \frac1\Delta \sum_{j=1}^J \E^{\varphi^0} \left[ h_j \left( X_0^0, X_\Delta^0, X_0^0; A \right) \right] \right) \neq 0$,
\item $\det \left( \widetilde{\mathfrak h}_J^0(A) \right) \neq 0$,
\item $\det \left( \widehat{\mathfrak h}_J^0(A) \right) \neq 0$.
\end{enumerate}
Since the proofs of the two lemmas are similar we only write the details of the first one.

\begin{lemma} \label{lem:dini_tilde}
Under Assumption \ref{ass:beta_functions} and Assumption \ref{ass:for_Dini} there exists $\epl_0>0$ such that for all $0<\epl<\epl_0$ there exists $\widetilde \gamma = \widetilde \gamma(\epl)$ such that if $\Delta$ is independent of $\epl$ or $\Delta=\epl^\zeta$ with $\zeta>0$ and $\zeta\neq1$, $\zeta\neq2$
\begin{equation}
\widetilde{\mathcal G}_J(\epl, A + \widetilde \gamma(\epl)) = 0 \qquad \text{and} \qquad \det \left( \widetilde{\mathcal H}_J(\epl, A + \widetilde \gamma(\epl)) \right) \neq 0.
\end{equation}
Moreover
\begin{equation}
\lim_{\epl\to0} \widetilde \gamma(\epl) = 0.
\end{equation}
\end{lemma}
\begin{proof}
Let us first extend the functions $\widetilde{\mathcal G}_J$ and $\widetilde{\mathcal H}_J$ by continuity in $\epl=0$ with their limit given by Proposition \ref{pro:continuity_G_tilde_0_D} and Proposition \ref{pro:continuity_G_tilde_0_De} depending on $\Delta$ and note that due to \eqref{eq:limit_score_A_ind} if $\Delta$ is independent of $\epl$ and \eqref{eq:limit_score_A_e} otherwise, we have
\begin{equation}
\widetilde{\mathcal G}_J(0,A) = 0.
\end{equation}
Moreover, by \eqref{eq:limit_diff_score_tilde_A_ind}, \eqref{eq:limit_diff_score_tilde_A_e} and Assumption \ref{ass:for_Dini}, we know that
\begin{equation}
\det \left( \widetilde{\mathcal H}_J(0,A) \right) \neq 0.
\end{equation}
Therefore, since the functions $\widetilde{\mathcal G}_J$ and $\widetilde{\mathcal H}_J$ are continuous by Proposition \ref{pro:continuity_G_tilde}, the implicit function theorem (see \cite[Theorem 2]{HuR03}) gives the desired result.
\end{proof}

\begin{lemma} \label{lem:dini_hat}
Under Assumption \ref{ass:beta_functions} and Assumption \ref{ass:for_Dini} there exists $\epl_0>0$ such that for all $0<\epl<\epl_0$ there exists $\widehat \gamma = \widehat \gamma(\epl)$ such that if $\Delta$ is independent of $\epl$ or $\Delta = \epl^\zeta$ with $\zeta \in (0,1)$
\begin{equation}
\widehat{\mathcal G}_J(\epl, A + \widehat \gamma(\epl)) = 0 \qquad \text{and} \qquad \det \left( \widehat{\mathcal H}_J(\epl, A + \widehat \gamma(\epl)) \right) \neq 0.
\end{equation}
Moreover
\begin{equation}
\lim_{\epl\to0} \widehat \gamma(\epl) = 0.
\end{equation}
\end{lemma}

We are now ready to prove the asymptotic unbiasedess of the estimators, i.e., Theorem \ref{thm:unbiasedness_hat} and Theorem \ref{thm:unbiasedness_tilde}. We only prove Theorem \ref{thm:unbiasedness_tilde} for the estimator $\widetilde A^\epl_{N,J}$ with filtered data. The proof of Theorem \ref{thm:unbiasedness_hat} for the estimator $\widehat A^\epl_{N,J}$ without filtered data is analogous and is omitted here.

\begin{proof}[Proof of Theorem \ref{thm:unbiasedness_tilde}]
We need to show for a fixed $0 < \epl < \epl_0$:
\begin{enumerate}
\item the existence of the solution $\widetilde A^\epl_{N,J}$ of the system $\widetilde G^\epl_{N,J}(a) = 0$ with probability tending to one as $N \to \infty$;
\item $\lim_{N \to \infty} \widetilde A^\epl_{N,J} = A + \widetilde \gamma(\epl)$ in probability with $\lim_{\epl \to 0} \widetilde \gamma(\epl) = 0$.
\end{enumerate}
We first note that by Lemma \ref{lem:dini_tilde} we have
\begin{equation}
\lim_{\epl \to 0} \widetilde \gamma(\epl) = 0.
\end{equation}
We then follow the steps of the proof of \cite[Theorem 3.2]{BiS95}. Due to \cite[Theorem A.1]{BaS94}, claims (i) and (ii) hold true if we verify that
\begin{equation} \label{eq:final_prove_result1}
\lim_{N \to \infty} \sup_{a \in B^\epl_{C,N}} \norm{\frac1N \dot{\widetilde G}^\epl_{N,J}(a) - \widetilde{\mathcal H}_J(\epl, A + \widetilde \gamma(\epl))} = 0, \qquad \text{in probability},
\end{equation}
and as $N \to \infty$
\begin{equation} \label{eq:final_prove_result2}
\frac{1}{\sqrt N} \widetilde G^\epl_{N,J}(A + \widetilde \gamma(\epl)) \to \mathcal N \left( 0, \Lambda^\epl \right), \qquad \text{in law},
\end{equation}
where $\Lambda^\epl$ is a positive definite covariance matrix and
\begin{equation}
B_{C,N}^\epl = \left\{ a \in \mathcal A \colon \norm{a - (A + \widetilde \gamma(\epl))} \le \frac{C}{\sqrt{N}} \right\},
\end{equation}
for $C>0$ small enough such that $B_{C,1} \subset \mathcal A$. Result \eqref{eq:final_prove_result2} is a consequence of \cite[Theorem 1]{Flo89}. We then have
\begin{equation}
\begin{aligned}
\sup_{a \in B^\epl_{C,N}} \norm{\frac1N \dot{\widetilde G}^\epl_{N,J}(a) - \widetilde{\mathcal H}_J(\epl, A + \widetilde \gamma(\epl))} &\le \sup_{a \in B^\epl_{C,1}} \norm{\frac{1}{N\Delta} \sum_{i=0}^{N-1} \sum_{j=1}^J h_j(\widetilde X_n^\epl, \widetilde X_{n+1}^\epl, \widetilde Z^\epl_n; a) - \widetilde{\mathcal H}_J(\epl, a)} \\
&\quad + \sup_{a \in B^\epl_{C,N}} \norm{\widetilde{\mathcal H}_J(\epl, a) - \widetilde{\mathcal H}_J(\epl, A + \widetilde \gamma(\epl))},
\end{aligned}
\end{equation}
where the right-hand side vanishes by \cite[Lemma 3.3]{BiS95} and the continuity of $\widetilde{\mathcal H}$ (Proposition \ref{pro:continuity_G_tilde}), implying result \eqref{eq:final_prove_result1}. Hence, we proved (i) and (ii), which conclude the proof of the theorem.
\end{proof}

\begin{remark} \label{rem:biasedness}
Notice that if $\Delta=\epl^\zeta$ with $\zeta>2$ and we do not employ the filter, in view of \eqref{eq:continuity_G_hat_0_De_different} and following the same proof of Theorem \ref{thm:unbiasedness_tilde}, we could compute the asymptotic limit of $\widehat A_{N,J}^\epl$ as $N$ goes to infinity and $\epl$ vanishes if we knew $a^*$ such that
\begin{equation}
\sum_{j=1}^J \E^{\varphi^0} \left[ \beta_j(X_0^0;a^*) \left( \sigma \phi_j''(X_0^0;a^*) - \alpha \cdot V'(X_0^0) \phi_j'(X_0^0;a^*) + \lambda_j(a^*) \phi_j(X_0^0;a^*) \right) \right] = 0.
\end{equation}
The value of $a^*$ can not be found analytically since it is, in general, different from the drift coefficients $\alpha$ and $A$ of the multiscale and homogenized equations \eqref{eq:SDE_MS} and \eqref{eq:SDE_H}. Nevertheless, we observe that in the simple scale of the multiscale Ornstein-Uhlenbeck process we have $a^* = \alpha$.
\end{remark}

\section{Conclusion}

In this work we presented new estimators for learning the effective drift coefficient of the homogenized Langevin dynamics when we are given discrete observations from the original multiscale diffusion process. Our approach relies on a martingale estimating function based on the eigenvalues and eigenfunctions of the generator of the coarse-grained model and on a linear time-invariant filter from the exponential family, which is employed to smooth the original data. We studied theoretically the convergence properties of our estimators when the sample size goes to infinity and the multiscale parameter describing the fastest scale vanishes. In Theorem \ref{thm:unbiasedness_hat} and Theorem \ref{thm:unbiasedness_tilde} we proved respectively the asymptotic unbiasedness of the estimators with and without filtered data. We remark that the former is not robust with respect to the sampling rate at finite multiscale parameter while the estimator with filtered data is robust independently of the sampling rate. We analysed numerically the dependence of our estimators on the number of observations and the number of eigenfunctions employed in the estimating function noticing that the first eigenvalues in magnitude are sufficient to approximate the drift coefficient. Moreover, we performed several numerical experiments, which highlighted the effectiveness of our approach and confirmed our theoretical results. We believe that eigenfunction estimators can be very useful in applications, for example to multiparticle systems and their mean field limit \cite{GoP18}, since the eigenvalue problem for the generator of a reversible Markov process is a very well studied problem. This means, in particular, that it is possible to study rigorously the proposed estimators and to prove asymptotic unbiasedness and asymptotic normality. Furthermore, in order to be able to assess the accuracy of the estimators, we could analyse its rate of convergence with respect to both the number of observations and the fastest scale. This is a highly nontrivial problem since it first requires the development of a fully quantitative periodic homogenization theory and we will return to this problem in future work. Finally, we think that it would be interesting to extend our estimators to the non-parametric framework and consider more general multiscale models.

\subsection*{Acknowledgements} 

We thank the anonymous referees for useful comments and suggestions.

AA and AZ are partially supported by the Swiss National Science Foundation, under grant No. 200020\_172710. The work of GAP was partially funded by the EPSRC, grant number EP/P031587/1, and by JPMorgan Chase \& Co. Any views or opinions expressed herein are solely those of the authors listed, and may differ from the views and opinions expressed by JPMorgan Chase \& Co. or its affiliates. This material is not a product of the Research Department of J.P. Morgan Securities LLC. This material does not constitute a solicitation or offer in any jurisdiction.

\begin{appendices}
	
\section{Technical results}

In this section we prove some technical results which are used to show the unbiasedness of the estimators $\widehat A^\epl_{N,J}$ and $\widetilde A^\epl_{N,J}$. We first study the properties of the filter applied to discrete data and then we focus on the regularity of the eigenfunctions and eigenvalues of the generator. We finally prove a formula which can be interpreted as an approximation of the Itô's lemma.

\subsection{Application of the filter to discrete data} \label{app:technical_results}

The following result quantifies the expected distance among the continuous process $Z^\epl_t$ and the filtered observations $\widetilde Z^\epl_n$.

\begin{lemma} \label{lem:distance_Z_Ztilde}
Let $0<\Delta<1$, $N$ be a positive integer and let $\widetilde Z_n^\epl$ and $Z_t^\epl$ be defined respectively in \eqref{eq:Z_tilde} and \eqref{eq:Z} with $\widetilde X^\epl_0 = X^\epl_0$ distributed according to its invariant measure $\varphi^\epl$. Then there exists a constant $C>0$ independent of $\epl$, $\Delta$ and $N$ such that for all $n = 0, \dots, N$ and for all $p\ge1$
\begin{equation}
\left( \E^{\varphi^\epl} \abs{Z_{n\Delta}^\epl - \widetilde Z_n^\epl}^p \right)^{1/p} \le C \left( \Delta^{1/2} + \min \left\{ \epl, \Delta\epl^{-1} \right\} \right),
\end{equation}
where $\E^{\varphi^\epl}$ denotes the expectation with respect to the Wiener measure and the fact that $X^\epl_0$ is distributed according to $\varphi^\epl$.
\end{lemma}
\begin{proof}
In order to simplify the notation, let us define the quantity
\begin{equation}
E \defeq \E^{\varphi^\epl}\abs{Z_{n\Delta}^\epl - \widetilde Z_n^\epl}^p,
\end{equation}
which is equivalent to
\begin{equation}
E = \E^{\varphi^\epl} \abs{\sum_{k=0}^{n-1} \int_{k\Delta}^{(k+1)\Delta} \left( e^{-(n\Delta - s)} X_s^\epl - e^{-\Delta(n-k)} \widetilde X_k^\epl \right) \dd s}^p.
\end{equation}
Then by Jensen's inequality applied to the convex function $y \mapsto \abs{y}^p$ and since $X^\epl_{k\Delta} = \widetilde X^\epl_k$ we have
\begin{equation} \label{eq:decomposition_E}
\begin{aligned}
E &\le 2^{p-1} \E^{\varphi^\epl} \left( \sum_{k=0}^{n-1} \int_{k\Delta}^{(k+1)\Delta} e^{-(n\Delta - s)} \abs{X_s^\epl - X_{k\Delta}^\epl} \dd s \right)^p \\
&\qquad + 2^{p-1} \E^{\varphi^\epl} \left( \sum_{k=0}^{n-1} \int_{k\Delta}^{(k+1)\Delta} \left( e^{-(n\Delta-s)} - e^{-\Delta(n-k)} \right) \dd s \abs{\widetilde X_k^\epl} \right)^p \\
&\eqdef 2^{p-1} \left( E_1 + E_2 \right).
\end{aligned}
\end{equation}
We now study the two terms separately. Applying \cite[Lemma B.1]{AGP20} we first get
\begin{equation} \label{eq:bound_E1_intermediate}
\begin{aligned}
E_1 &= \E^{\varphi^\epl} \left( \int_0^{n\Delta} e^{-(n\Delta-s)} \abs{X_s^\epl - \sum_{k=0}^{n-1} X_{k\Delta}^\epl \chi_{[k\Delta, (k+1)\Delta)}(s)} \dd s \right)^p \\
&\le \int_0^{n\Delta} e^{-(n\Delta-s)} \E^{\varphi^\epl} \abs{X_s^\epl - \sum_{k=0}^{n-1} X_{k\Delta}^\epl \chi_{[k\Delta, (k+1)\Delta)}(s)}^p \dd s,
\end{aligned}
\end{equation}
and, in order to bound the term inside the integral, we can follow two different procedures. Either we employ \cite[Lemma 6.1]{PaS07}, which gives
\begin{equation} \label{eq:bound_E1_first}
\E^{\varphi^\epl} \abs{X_s^\epl - \sum_{k=0}^{n-1} X_{k\Delta}^\epl \chi_{[k\Delta, (k+1)\Delta)}(s)}^p \le C \left( \Delta^p + \Delta^{p/2} + \epl^p \right),
\end{equation}
where $C > 0$ is a constant independent of $\epl$ and $\Delta$ or we notice that, since $X_t^\epl$ has bounded moments of any order by \cite[Corollary 5.4]{PaS07} and $p$ is bounded, it holds for all $s \in [k\Delta, (k+1)\Delta)$
\begin{equation} \label{eq:bound_E1_second}
\begin{aligned}
\E^{\varphi^\epl} \abs{X_s^\epl - X_{k\Delta}^\epl}^p &= \E^{\varphi^\epl} \abs{-\alpha \int_{k\Delta}^s X_r^\epl \dd r - \frac1\epl \int_{k\Delta}^s p'\left( \frac{X_r^\epl}{\epl} \right) \dd r + \sqrt{2\sigma}W_s }^p \\
&\le C \left( \Delta^p + \Delta^p \epl^{-p} + \Delta^{p/2} \right).
\end{aligned}
\end{equation}
Therefore, due to \eqref{eq:bound_E1_intermediate}, \eqref{eq:bound_E1_first} and \eqref{eq:bound_E1_second}, we obtain
\begin{equation} \label{eq:bound_E1}
E_1 \le C \left( \Delta^{p/2} + \min \left\{ \epl^p, \Delta^p \epl^{-p} \right\} \right).
\end{equation}
Let us now consider $E_2$, which can be first bounded by
\begin{equation}
E_2 \le \Delta^p \E^{\varphi^\epl} \left( \sum_{k=0}^{n-1} \left( e^{-\Delta(n-1-k)} - e^{-\Delta(n-k)} \right) \abs{\widetilde X_k^\epl} \right)^p,
\end{equation}
and note that
\begin{equation}
\sum_{k=0}^{n-1} \left( e^{-\Delta(n-1-k)} - e^{-\Delta(n-k)} \right) = \sum_{k=0}^{n-1} \left( e^{-\Delta k} - e^{-\Delta(k+1)} \right) = 1 - e^{-\Delta n}.
\end{equation}
Therefore, applying Jensen's inequality and due to the fact that $\widetilde X_k^\epl$ has bounded moments of any order by \cite[Corollary 5.4]{PaS07} we have 
\begin{equation} \label{eq:bound_E2}
E_2 \le \Delta^p (1 - e^{-\Delta n})^{p-1} \sum_{k=0}^{n-1} \left( e^{-\Delta(n-1-k)} - e^{-\Delta(n-k)} \right) \E^{\varphi^\epl} \abs{\widetilde X_k^\epl}^p \le C \Delta^p,
\end{equation}
which, together with \eqref{eq:decomposition_E} and \eqref{eq:bound_E1}, gives the desired result.
\end{proof}

We now show the ergodicity of the process $(\widetilde X^\epl_n, \widetilde Z^\epl_n)$, where the first component is a sample from the continuous-time process, i.e. $\widetilde X^\epl_n = X_{n\Delta}^\epl$, while the second component is computed starting from the discrete observations $\widetilde X^\epl_n$.

\begin{lemma} \label{lem:ergodicity_XZ_tilde}
Let $\Delta>0$ and let Assumption \ref{ass:dissipative_setting} hold. Then the couple $(\widetilde X_n^\epl, \widetilde Z_n^\epl)$, where $\widetilde X_n^\epl$ is a sample from the continuous process \eqref{eq:SDE_MS} and $\widetilde Z_n^\epl$ is defined in \eqref{eq:Z_tilde}, admits a unique invariant measure with density with respect to the Lebesgue measure denoted by $\widetilde \rho^\epl = \widetilde \rho^\epl(x,z)$. Moreover, if $\Delta$ is independent of $\epl$, it converges in law to the two-dimensional process $(\widetilde X_n^0, \widetilde Z_n^0)$ with $\widetilde \rho^0 = \widetilde \rho^0(x,z)$ as density of the invariant measure.
\end{lemma}
\begin{proof}
By definition \eqref{eq:Z_tilde_1} we obtain the following stochastic difference equation
\begin{equation} \label{eq:diff_eq_tilde}
\widetilde Z^\epl_{n+1} = e^{-\Delta} \widetilde Z^\epl_n + \Delta e^{-\Delta} \widetilde X^\epl_n,
\end{equation}
where $\widetilde X^\epl_n$ is a stationary and ergodic sequence. Observing that $\log e^{-\Delta} = -\Delta < 0$, applying Theorem 1 and in view of Remark 1.3 in \cite{Bra86} we deduce the existence of a unique invariant measure for the couple $(\widetilde X_n^\epl, \widetilde Z_n^\epl)$. Let us notice that in the theorem the sequence $\widetilde X_n^\epl$ must be defined for all $n\in\Z$ while in our framework $n\in\N$, but let us also remark that any stationary process indexed by $\N$ can be extended to one indexed by $\Z$ in an essentially unique way. Moreover, if $\Delta$ is independent of $\epl$, the same reasoning can be repeated to get the existence of a unique invariant measure for the couple $(\widetilde X_n^0, \widetilde Z_n^0)$. Finally, standard homogenization theory implies the weak convergence of $\widetilde \rho^\epl$ to $\widetilde \rho^0$, which concludes the proof.
\end{proof}

Let us now denote the marginal invariant distributions of $Z^\epl, \widetilde Z^\epl$ and $Z^0, \widetilde Z^0$ respectively by $\psi^\epl = \psi^\epl(z), \widetilde \psi^\epl = \widetilde \psi^\epl(z)$ and $\psi^0 = \psi^0(z), \widetilde \psi^0 = \widetilde \psi^0(z)$.

\begin{corollary} \label{cor:asymptotic_distances_Z}
Let $Z^\epl$ and $\widetilde Z^\epl$ be distributed respectively according to $\psi^\epl$ and $\widetilde \psi^\epl$. Then there exists a constant $C>0$ independent of $\epl$ and $\Delta$ such that
\begin{equation}
\left( \E \abs{Z^\epl - \widetilde Z^\epl}^p \right)^{1/p} \le C \left( \Delta^{1/2} + \min \left\{ \epl, \Delta\epl^{-1} \right\} \right).
\end{equation}
\end{corollary}
\begin{proof}
The result follows directly from Lemma \ref{lem:distance_Z_Ztilde} by letting $n$ go to infinity, noting that the constant $C$ is independent of $n$ and employing ergodicity given by Lemma \ref{lem:ergodicity_XZ_tilde}.
\end{proof}

It directly follows that $\widetilde Z_n^\epl$ has bounded moments of all order and, in particular, we can prove Lemma \ref{lem:Ztilde_bounded_moments}.

\begin{proof}[Proof of Lemma \ref{lem:Ztilde_bounded_moments}]
Applying Jensen's inequality to the function $x\mapsto\abs{x}^p$, we have
\begin{equation}
\E^{\widetilde \rho^\epl} \abs{\widetilde Z^\epl}^p \le 2^{p-1} \E^{\rho^\epl} \abs{Z^\epl}^p + 2^{p-1} \E \abs{\widetilde Z^\epl - Z^\epl}^p,
\end{equation}
then bounding the two terms in the right-hand side respectively with \cite[Lemma C.1]{AGP20} and Corollary \ref{cor:asymptotic_distances_Z} gives the desired result.
\end{proof}

\subsection{Properties of eigenfunctions and eigenvalues of the generator}

Let us now consider the eigenvalue and the eigenfunctions of the generator of SDE \eqref{eq:SDE_H_a}.

\begin{lemma} \label{lem:regularity_eigen_a}
Let $\{(\lambda_j(a),\phi_j(\cdot;a))\}_{j=0}^\infty$ be the solutions of the eigenvalue problem \eqref{eq:eigen_problem_general}. Then $\phi_j(x;a)$ and $\lambda_j(a)$ are continuously differentiable with respect to $a$ for all $x\in\R$ and for all $j\in\N$. Moreover, $\phi_j(\cdot;a)$ and $\dot \phi_j(\cdot;a)$ belong to $\mathcal C^\infty(\R)$.
\end{lemma}
\begin{proof}
The first result follows from Section 2 and Section 6 in \cite{Sch74}. Let us remark that the fact that the spectrum is discrete and non-degenerate is guaranteed by \cite[Section 4.7]{Pav14}. Finally, the second result in the statement is a direct consequence of the elliptic regularity theory.
\end{proof}

\subsection{Approximation of the Itô formula} \label{app:approximation_formula}

In this section we prove Lemma \ref{lem:expansion_fXD}, which is an approximation of the Itô's lemma applied to the stochastic process $X_t^\epl$. Let us introduce the process $S_t^\epl$ defined by the following SDE with initial condition $S_0^\epl = X_0^\epl$
\begin{equation} \label{eq:def_S}
\d S^\epl_t = - \alpha V'(X_t^\epl) (1 + \Phi'(Y_t^\epl)) \d t + \sqrt{2\sigma} (1+\Phi'(Y_t^\epl)) \d W_t,
\end{equation}
where $Y_t^\epl = X_t^\epl/\epl$ and $\Phi$ is the cell function which solves equation \eqref{eq:cell_problem}, and notice that
\begin{equation}
S^\epl_\Delta = X^\epl_0 - \alpha \int_0^\Delta V'(X_t^\epl) (1 + \Phi'(Y_t^\epl)) \dd t + \sqrt{2\sigma} \int_0^\Delta (1+\Phi'(Y_t^\epl)) \dd W_t.
\end{equation}
Therefore, due to equation (5.7) in \cite{PaS07} we have
\begin{equation}
\abs{X_\Delta^\epl - S_\Delta^\epl} = \epl \abs{\Phi(Y^\epl_\Delta) - \Phi(Y^\epl_0)},
\end{equation}
and, since $\Phi$ is bounded by \cite[Lemma 5.5]{PaS07}, we get for a constant $C>0$ independent of $\Delta$ and $\epl$
\begin{equation} \label{eq:approx_S}
\abs{X_\Delta^\epl - S_\Delta^\epl} \le C \epl.
\end{equation}
Before showing the main formula, we need two preliminary estimates which will be employed later in the analysis. The proofs of Lemma \ref{lem:approx_f1} and Lemma \ref{lem:approx_f2} are inspired by the proof of Proposition 5.8 in \cite{PaS07}.

\begin{lemma} \label{lem:approx_f1}
Let $f \colon \R \to \R$ be a continuously differentiable function such that $f,f'$ are polynomially bounded. Then
\begin{equation} \label{eq:approx_f1_statement}
\int_0^\Delta \alpha \cdot V'(X_t^\epl) f(X_t^\epl) (1 + \Phi'(Y_t^\epl)) \dd t = A \cdot V'(X_0^\epl) f(X_0^\epl) \Delta + R_1(\epl,\Delta),
\end{equation}
where the remainder satisfies for all $p\ge1$ and for a constant $C>0$ independent of $\Delta$ and $\epl$
\begin{equation}
\left( \E^{\varphi^\epl} \abs{R_1(\epl,\Delta)}^p \right)^{1/p} \le C(\epl^2 + \Delta^{1/2}\epl + \Delta^{3/2}).
\end{equation}
\end{lemma}
\begin{proof}
To obtain the remainder $R_1(\epl,\Delta)$ we decompose suitably the difference between the left-hand side and the right-hand side of \eqref{eq:approx_f1_statement}. Applying Jensen's inequality to the function $z\mapsto\abs{z}^p$ we have
\begin{equation} \label{eq:decomposition_I123}
\begin{aligned}
\E^{\varphi^\epl} \abs{R_1(\epl,\Delta)}^p &\le 3^{p-1} \E^{\varphi^\epl} \abs{\int_0^\Delta \alpha \cdot \left( V'(X_t^\epl) - V'(X_0^\epl) \right) f(X_t^\epl) (1 + \Phi'(Y_t^\epl)) \dd t}^p \\
&\quad + 3^{p-1} \E^{\varphi^\epl} \abs{\alpha \cdot V'(X_0^\epl) \int_0^\Delta \left( f(X_t^\epl) - f(X_0^\epl) \right) (1 + \Phi'(Y_t^\epl)) \dd t}^p \\
&\quad + 3^{p-1} \E^{\varphi^\epl} \abs{f(X_0^\epl) V'(X_0^\epl) \cdot \int_0^\Delta \left( \alpha (1 + \Phi'(Y_t^\epl)) - A \right) \dd t}^p \\
&\eqdef I_1(\epl,\Delta) + I_2(\epl,\Delta) + I_3(\epl,\Delta).
\end{aligned}
\end{equation}
Letting $C>0$ be a constant independent of $\epl$ and $\delta$, we now bound the three terms separately. First, applying Hölder inequality and since $V'$ is Lipschitz, $\Phi'$ is bounded, $f$ is polynomially bounded and $X_t^\epl$ has bounded moments of any order by \cite[Corollary 5.4]{PaS07}, we have
\begin{equation}
I_1(\epl,\Delta) \le C \Delta^{p-1} \int_0^\Delta \E^{\varphi^\epl} \abs{X_t^\epl - X_0^\epl}^p \dd t,
\end{equation}
then applying \cite[Lemma 6.1]{PaS07} we obtain
\begin{equation} \label{eq:bound_I1}
I_1(\epl,\Delta) \le C \left( \Delta^{2p} + \Delta^{3p/2} + \epl^p\Delta^p \right).
\end{equation}
We then rewrite $I_2(\epl,\Delta)$ employing the mean value theorem
\begin{equation}
I_2(\epl,\Delta) = 3^{p-1} \E^{\varphi^\epl} \abs{\alpha \cdot V'(X_0^\epl) \int_0^\Delta f'(\widetilde X^\epl_t) (X_t^\epl - X_0^\epl) (1 + \Phi'(Y_t^\epl)) \dd t}^p,
\end{equation}
where $\widetilde X_t^\epl$ assumes values between $X_0^\epl$ and $X_t^\epl$, and we repeat the same reasoning as for $I_1(\epl,\Delta)$ to get
\begin{equation} \label{eq:bound_I2}
I_2(\epl,\Delta) \le C \left( \Delta^{2p} + \Delta^{3p/2} + \epl^p\Delta^p \right).
\end{equation}
We now consider the function
\begin{equation}
H(y) \eqdef \alpha (1+\Phi'(y)) - A,
\end{equation}
which by definition of $A$ and due to \eqref{eq:K_H} has zero mean with respect to $\mu$ defined in \eqref{eq:def_mu}. Therefore, since $f$ and $V'$ are polynomially bounded and $X_0^\epl$ has bounded moments of any order by \cite[Corollary 5.4]{PaS07}, applying \cite[Lemma 5.6]{PaS07} we obtain
\begin{equation} \label{eq:bound_I3}
I_3(\epl,\Delta) \le C \left( \epl^{2p} + \epl^p\Delta^p + \epl^p\Delta^{p/2} \right).
\end{equation}
Finally, for $\epl$ and $\Delta$ sufficiently small, the desired result follows from \eqref{eq:decomposition_I123} and from estimates \eqref{eq:bound_I1}, \eqref{eq:bound_I2} and \eqref{eq:bound_I3}.
\end{proof}

\begin{lemma} \label{lem:approx_f2}
Let $f \colon \R \to \R$ be a continuously differentiable function such that $f,f'$ are polynomially bounded. Then
\begin{equation} \label{eq:approx_f2_statement}
\int_0^\Delta \sigma f(X_t^\epl) (1 + \Phi'(Y_t^\epl))^2 \dd t = \Sigma f(X_0^\epl) \Delta + R_2(\epl,\Delta),
\end{equation}
where the remainder satisfies for all $p\ge1$ and for a constant $C>0$ independent of $\Delta$ and $\epl$
\begin{equation}
\left( \E^{\varphi^\epl} \abs{R_2(\epl,\Delta)}^p \right)^{1/p} \le C(\epl^2 + \Delta^{1/2}\epl + \Delta^{3/2}).
\end{equation}
\end{lemma}
\begin{proof}
To obtain the remainder $R_2(\epl,\Delta)$ we decompose suitably the difference between the left-hand side and the right-hand side of \eqref{eq:approx_f2_statement}. Applying Jensen's inequality to the function $z\mapsto\abs{z}^p$ we have
\begin{equation} \label{eq:decomposition_I12}
\begin{aligned}
\E^{\varphi^\epl} \abs{R_2(\epl,\Delta)}^p &\le 2^{p-1} \E^{\varphi^\epl} \abs{\int_0^\Delta \sigma \left( f(X_t^\epl) - f(X_0^\epl) \right) (1+\Phi'(Y_t^\epl)^2) \dd t}^p \\
&\quad + 2^{p-1} \E^{\varphi^\epl} \abs{f(X_0^\epl) \int_0^\Delta \left( \sigma (1+\Phi'(Y_t^\epl))^2 - \Sigma \right) \dd t}^p \\
&\eqdef I_1(\epl,\Delta) + I_2(\epl,\Delta).
\end{aligned}
\end{equation}
Letting $C>0$ be a constant independent of $\epl$ and $\Delta$, we now bound the two terms separately. First, we rewrite $I_1(\epl,\Delta)$ employing the mean value theorem
\begin{equation}
I_1(\epl,\Delta) = 2^{p-1} \E^{\varphi^\epl} \abs{\int_0^\Delta \sigma f'(\widetilde X_t^\epl) (X_t^\epl - X_0^\epl) (1+\Phi'(Y_t^\epl))^2 \dd t}^p,
\end{equation}
where $\widetilde X_t^\epl$ assumes values between $X_0^\epl$ and $X_t^\epl$, then applying Hölder inequality and since $\Phi'$ is bounded, $f'$ is polynomially bounded and $X_t^\epl$ has bounded moments of any order by \cite[Corollary 5.4]{PaS07}, we have
\begin{equation}
I_1(\epl,\Delta) \le C \Delta^{p-1} \int_0^\Delta \E^{\varphi^\epl} \abs{X_t^\epl - X_0^\epl}^p \dd t,
\end{equation}
and applying \cite[Lemma 6.1]{PaS07} we obtain
\begin{equation} \label{eq:bound_I1_f2}
I_1(\epl,\Delta) \le C \left( \Delta^{2p} + \Delta^{3p/2} + \epl^p\Delta^p \right).
\end{equation}
We now consider the function
\begin{equation}
H(y) \eqdef \sigma (1+\Phi'(y))^2 - \Sigma,
\end{equation}
which by definition of $\Sigma$ and due to \eqref{eq:K_H} has zero mean with respect to $\mu$ defined in \eqref{eq:def_mu}. Therefore, since $f$ is polynomially bounded and $X_0^\epl$ has bounded moments of any order by \cite[Corollary 5.4]{PaS07}, applying \cite[Lemma 5.6]{PaS07} we obtain
\begin{equation} \label{eq:bound_I2_f2}
I_2(\epl,\Delta) \le C \left( \epl^{2p} + \epl^p\Delta^p + \epl^p\Delta^{p/2} \right).
\end{equation}
Finally, for $\epl$ and $\Delta$ sufficiently small, the desired result follows from \eqref{eq:decomposition_I12} and from estimates \eqref{eq:bound_I1_f2} and \eqref{eq:bound_I2_f2}.
\end{proof}

We can now prove the main formula, which is employed repeatedly in the proof of the asymptotic unbiasedness of the drift estimators.

\begin{proof}[Proof of Lemma \ref{lem:expansion_fXD}]
Applying Itô's lemma to the process $S_t^\epl$ defined in \eqref{eq:def_S} with the function $f$ we have
\begin{equation}
\begin{aligned}
f(S_\Delta^\epl) &= f(X_0^\epl) - \int_0^\Delta \alpha \cdot V'(X_t^\epl) f'(X_t^\epl) (1 + \Phi'(Y_t^\epl)) \dd t + \int_0^\Delta \sigma f''(X_t^\epl) (1 + \Phi'(Y_t^\epl))^2 \dd t \\
&\quad + \sqrt{2\sigma} \int_0^\Delta f'(X_t^\epl) (1+\Phi'(Y_t^\epl)) \dd W_t,
\end{aligned}
\end{equation}
and due to Lemma \ref{lem:approx_f1} and Lemma \ref{lem:approx_f2} we obtain
\begin{equation}
\begin{aligned}
f(S_\Delta^\epl) &= f(X_0^\epl) - A \cdot V'(X_0^\epl) f'(X_0^\epl) \Delta + \Sigma f''(X_0^\epl) \Delta + \sqrt{2\sigma} \int_0^\Delta f'(X_t^\epl) (1+\Phi'(Y_t^\epl)) \dd W_t \\
&\quad - R_1(\epl,\Delta) + R_2(\epl,\Delta).
\end{aligned}
\end{equation}
Then we write
\begin{equation}
f(X_\Delta^\epl) = f(S_\Delta^\epl) + \left[ f(X_\Delta^\epl) - f(S_\Delta^\epl) \right] \eqdef f(S_\Delta^\epl) + R_3(\epl,\Delta),
\end{equation}
and, in order to conclude, it only remains to bound the expectation of $R_3(\epl,\Delta)$. Applying the mean value theorem and the Cauchy-Schwarz inequality and due to \eqref{eq:approx_S}, the hypotheses on $f$ and the fact that $X_t^\epl$ has bounded moments of any order by \cite[Corollary 5.4]{PaS07}, we obtain
\begin{equation}
\E^{\varphi^\epl} \abs{R_3(\epl,\Delta)}^p \le \left( \E^{\varphi^\epl} \abs{f'(\widetilde X)}^{2p} \right)^{1/2} \left( \E^{\varphi^\epl} \abs{X_\Delta^\epl - S_\Delta^\epl}^{2p} \right)^{1/2} \le C \epl^p,
\end{equation}
where $\widetilde X$ takes values between $X_\Delta^\epl$ and $S_\Delta^\epl$, and which together with the estimates for $R_1$ and $R_2$ implies the desired result.
\end{proof}
	
\section{Implementation details} \label{app:implementation}

In this section we present the main techniques that we employed in the implementation of the proposed method. The most important steps in the algorithm are the computation of the eigenvalues and eigenfunctions of the eigenvalue problem \eqref{eq:eigen_problem_equation}
\begin{equation}
\Sigma \phi_j''(x;a) - a \cdot V'(x) \phi_j'(x;a) + \lambda_j(a) \phi_j(x;a) = 0,
\end{equation}
and the solution of the non-linear system \eqref{eq:system2solve} or \eqref{eq:system2solve_filter} with filtered data. Let us first focus on the eigenvalue problem. We note that the domain of the eigenfunctions is the whole real line $\R$ and need to be truncated for numerical computations. We first consider the variational formulation of equation \eqref{eq:eigen_problem_equation}, i.e., we multiply it by $v \varphi_a$, where $v$ is a test function and $\varphi_a$ is the invariant distribution defined in \eqref{eq:phia_def}, and integrating by parts we obtain for all $j \in \N$ the following eigenvalue problem
\begin{equation}
\Sigma \int_\R \phi_j'(x;a) v'(x) \varphi_a(x) \dd x = \lambda_j(a) \int_\R \phi_j(x;a) v(x) \varphi_a(x) \dd x.
\end{equation}
Since $\varphi_a$ decays to zero exponentially fast, for all $\delta > 0$ there exists $r>0$ such that
\begin{equation}
\abs{\varphi_a(x)} < \delta \qquad \text{for all } x \not\in [-r,r].
\end{equation}
Hence, letting $R > 0$ we assume that $\varphi_a(\pm R) \simeq 0$ and we solve the truncated problem
\begin{equation} \label{eq:eigen_variational_truncated}
\Sigma \int_{-R}^{+R} \phi_j'(x;a) v'(x) \varphi_a(x) \dd x = \lambda_j(a) \int_{-R}^{+R} \phi_j(x;a) v(x) \varphi_a(x) \dd x.
\end{equation}
Notice that $R$ must be chosen big enough and such that
\begin{equation}
R \ge \max_{n=0,\dots,N} \max \left\{ \abs{\widetilde X_n^\epl}, \abs{\widetilde Z_n^\epl} \right\} \eqdef \bar R,
\end{equation}
and we take $R = \max \{ \bar R + 0.1, 1.7 \}$. Moreover, in order to have a unique solution for the eigenvector $\phi_j(\cdot;a)$ we impose the additional conditions
\begin{equation} \label{eq:condition_eigenfunction}
\phi_j(R;a) > 0 \qquad \text{and} \qquad \int_{-R}^{+R} \phi_j(x;a)^2 \varphi_a(x) \dd x = 1.
\end{equation}
We then introduce a partition $\mathcal T_h$ of $[-R,R]$ in $N_h$ subintervals $K_i = [x_{i-1}, x_i]$ with 
\begin{equation}
-R = x_0 < x_1 < \dots < x_{N_h} < x_{N_h} = +R,
\end{equation}
and $h = 2R/N_h$, and we construct the discrete space
\begin{equation}
X_h^1 = \left\{ v_h \in C^0([-R,+R]) \colon v_h |_{K_i} \in \mathbb P^1 \; \forall \; K_i \in \mathcal T_h \right\},
\end{equation}
which is constituted by continuous piecewise linear functions. Note that the discretization parameter $h$ is chosen to be $h = 0.1$ or $h = 0.05$. We pick the characteristic Lagrangian basis $\{ \psi_k \}_{k=0}^{N_h}$ of $X_h^1$ characterized by the following property
\begin{equation}
\psi_k(x_i) = \delta_{ik} \qquad \text{for all } i,k = 0, \dots, N_h,
\end{equation}
where $\delta_{ik}$ is the Kronecker delta. We want to find $\phi_j(\cdot;a) \in X_h^1$ such that equation \eqref{eq:eigen_variational_truncated} holds true for all $v \in X_h^1$. Therefore, in equation \eqref{eq:eigen_variational_truncated} we substitute
\begin{equation}
\phi_j(x;a) = \sum_{k=0}^{N_h} \theta_j^{(k)}(a) \psi_k(x) \qquad \text{and} \qquad v(x) = \psi_i(x) \text{ for all } i = 0, \dots, N_h,
\end{equation}
and we obtain the discrete formulation
\begin{equation} \label{eq:gep}
S \Theta_j(a) = \lambda_j(a) M \Theta_j(a),
\end{equation}
where $\Theta_j(a) \in \R^{N_h+1}$ is such that $(\Theta_j(a))_k = \theta_j^{(k-1)}(a)$ and the components of the matrices $S,M \in \R^{N_h+1 \times N_h+1}$ are given by
\begin{equation}
S_{ik} = \Sigma \int_{-R}^{+R} \psi_{i-1}'(x) \psi_{k-1}'(x) \varphi_a(x) \dd x, \qquad \text{and} \qquad M_{ik} = \int_{-R}^{+R} \psi_{i-1}(x) \psi_{k-1}(x) \varphi_a(x) \dd x,
\end{equation}
where the integrals are approximated through the composite Simpson's quadrature rule. Equation \eqref{eq:gep} is a generalized eigenvalue problem which can be solved in \textsc{Matlab} using the function \texttt{eigs} or in \textsc{Phyton} using the function \texttt{scipy.sparse.linalg.eigsh}. Then we normalize $\Theta_j(a)$ or change its sign in order to impose the conditions \eqref{eq:condition_eigenfunction}, which can be rewritten as
\begin{equation}
\theta_j^{(N_h)}(a) > 0 \qquad \text{and} \qquad \Theta_j(a)^\top M \Theta_j(a) = 1.
\end{equation}
Once we compute $\lambda_j(a)$ and $\Theta_j(a)$ we have an approximation of the eigenvalues and eigenfunctions and we can construct the function $\widehat G^\epl_{N,J}(a)$ in \eqref{eq:score_function_NOfilter} or $\widetilde G^\epl_{N,J}(a)$ in \eqref{eq:score_function_YESfilter} with filtered data. Hence, it only remains to solve systems \eqref{eq:system2solve} or \eqref{eq:system2solve_filter}, i.e.,
\begin{equation}
\widehat G^\epl_{N,J}(a) = 0, \qquad \text{or} \qquad \widetilde G^\epl_{N,J}(a) = 0.
\end{equation}
To solve these equations we can follow two approaches:
\begin{itemize}
\item find the zero of $\widehat G^\epl_{N,J}(a)$ or $\widetilde G^\epl_{N,J}(a)$;
\item find the minimum of $\norm{\widehat G^\epl_{N,J}(a)}$ or $\norm{\widetilde G^\epl_{N,J}(a)}$.
\end{itemize}
In practice, for the first approach the function \texttt{fsolve} in \textsc{Matlab} or the function \texttt{scipy.optimize.fsolve} in \textsc{Python} can be used, while for the second one the function \texttt{fmincon} in \textsc{Matlab} or the function \texttt{scipy.optimize.minimize} in \textsc{Python} can be used. Finally, note that the functions implemented in \textsc{Matlab} or \textsc{Python} have been employed with their default parameters.

\section{Multidimensional diffusion processes} \label{app:multidimensional}

In this section we present how our methodology for estimating the drift coefficient of the homogenized equation can be extended to the case of multidimensional multiscale diffusion processes in $\R^d$. In the $d$-dimensional case the multiscale SDE \eqref{eq:SDE_MS} reads
\begin{equation}
\d X_t^\epl = - \sum_{m=1}^{M} \alpha_m \nabla V_m(X_t^\epl) \dd t - \frac1\epl \nabla p\left(\frac{X_t^\epl}\epl\right) \dd t + \sqrt{2\sigma} \dd W_t,
\end{equation}
where $W_t$ is a standard $d$-dimensional Brownian motion. The theory of homogenization (see e.g. \cite[Chapter 3]{BLP78} or \cite[Chapter 18]{PaS08}) then guarantees the existence of the homogenized SDE
\begin{equation}
\d X_t^0 = - \sum_{m=1}^{M} A_m \nabla V_m(X_t^0) \dd t + \sqrt{2\Sigma} \dd W_t,
\end{equation}
where $A_m,\Sigma \in \R^{d \times d}$ are given by $A_m = \alpha_m K$ and $\Sigma = \sigma K$. The matrix $K \in \R^{d \times d}$ is defined by
\begin{equation}
K = \int_{[0,L]^d} (I + \nabla \Phi(y)) (I + \nabla \Phi(y))^T \mu(dy) = \int_{[0,L]^d} (I + \nabla \Phi(y)) \mu(dy),
\end{equation}
where 
\begin{equation}
\mu(\d y) = \frac{1}{C_\sigma} e^{-p(y)/\sigma} \dd y \quad\text{with} \quad C_\sigma = \int_{[0,L]^d} e^{-p(y)/\sigma} \dd y,
\end{equation}
and where the function $\Phi \colon [0,L]^d \to \R^d$ is the unique solution with zero-mean with respect to the measure $\mu$ of the cell problem in $[0,L]^d$
\begin{equation}
- \nabla \Phi \nabla p + \sigma \Delta \Phi = \nabla p,
\end{equation}
endowed with periodic boundary conditions. Using the tensor notation, we can then define the drift coefficient $A \in \R^{M \times d \times d}$, which collects together the $M$ matrices $A_m$ for $m = 1, \dots, M$. Our goal is now to estimate the tensor $A$ and thus we need to define the score functions. First, the $d$-dimensional eigenvalue problem for $j = 1, \dots, J$ corresponding to \eqref{eq:eigen_problem_equation} is
\begin{equation}
\Sigma : \nabla^2 \phi_j(x;a) - \left( \sum_{m=1}^M a_m \nabla V_m(x) \right) \cdot \nabla \phi_j(x;a) + \lambda_j(a) \phi_j(x;a) = 0,
\end{equation}
where $:$ denotes the Frobenius inner product, $\nabla^2$ the Hessian matrix and the parameter $a \in \R^{M \times d \times d}$ collects together the $M$ matrices $a_m$ for $m = 1, \dots, M$. Then, in order to define the martingale estimating functions $g_j$ for $j = 1, \dots, J$, we take a collection $\{ \beta_j \}_{j=1}^J$ of functions $\beta_j(\cdot; a) \colon \R^d \to \R^{M \times d \times d}$ and we use equation \eqref{eq:def_g}. Finally, we construct the score functions $\widehat G^\epl_{N,J}$ and $\widetilde G^\epl_{N,J}$ in the same way as we did in the one dimensional case, i.e., employing equations \eqref{eq:score_function_NOfilter} and \eqref{eq:score_function_YESfilter}. We remark that the filtered data are obtained as in equation \eqref{eq:Z_tilde} by applying the filter component-wise. We can now compute the estimators $\widehat A^\epl_{N,J}$ and $\widetilde A^\epl_{N,J}$ by solving the nonlinear systems
\begin{equation}
\widehat G^\epl_{N,J}(a) = 0 \qquad \text{and} \qquad \widetilde G^\epl_{N,J}(a) = 0,
\end{equation}
which have dimension $Md^2$. From a theoretical point of view, slight modifications of the proofs allow to conclude that analogous results to the main theorems hold true, i.e., that the estimators are asymptotically unbiased in the limit of infinite observations and when the multiscale parameter vanishes. However, the problem becomes more complex and computationally expensive from a numerical viewpoint, in particular when the dimension $d$ is large. In fact, the final nonlinear system, which has to be solved, has dimension $Md^2$ instead of $M$ and, most importantly, it is required to solve the eigenvalue problem for the generator of a diffusion process in $d$ dimensions.

\end{appendices}

\bibliographystyle{siamnodash}
\bibliography{biblio}

\end{document}